\documentclass[letterpaper,10pt]{article}
\usepackage{amsfonts,amssymb,amsthm,eucal,amsmath}
\usepackage{graphicx,pdflscape}
\usepackage[T1]{fontenc}
\usepackage{latexsym,url}
\usepackage[all]{xy}

\newtheorem{thm}{Theorem}[section]
\newtheorem{lemma}[thm]{Lemma}

\newtheorem{prop}[thm]{Proposition}

\theoremstyle{definition}

\theoremstyle{remark}
\newtheorem{rmk}[thm]{Remark}
\newcommand{\N}{\mathbb{N}}
\newcommand{\Z}{\mathbb{Z}}

\newcommand{\bdry}{\partial}

\title{Incompressible surfaces in handlebodies and boundary reducible 3-manifolds}

\author{Jo\~{a}o Miguel Nogueira\footnote{Author partially supported by a Fulbright PhD grant and Calouste Gulbenkian Foundation PhD grant 79084.}\; and Henry Segerman\footnote{Author partially supported by an NSF RTG postdoc.}} 

\date{}

\begin{document}

\maketitle

\begin{abstract}
We study the existence of incompressible embeddings of surfaces into the genus two handlebody. We show that for every compact surface with boundary, orientable or not, there is an incompressible embedding of the surface into the genus two handlebody. In the orientable case the embedding can be either separating or non-separating. We also consider the case in which the genus two handlebody is replaced by an orientable $3$-manifold with a compressible boundary component of genus greater than or equal to two. 
\end{abstract}

\section{Introduction}

In this paper we study the existence of incompressible, but generally boundary compressible, embeddings of compact surfaces with boundary in a boundary reducible 3-manifold\footnote{A boundary reducible 3-manifold is a 3-manifold with a compressible boundary component.} with a compressible boundary component of genus greater than or equal to two. In particular we study the case in which the 3-manifold is a handlebody of genus two, which we denote by $H_2$.\footnote{Note that the disk is the only surface with an incompressible and boundary incompressible embedding in a handlebody.} This naturally extends to the case of a handlebody of genus greater than two.\\

This type of question was first studied by Jaco \cite{Jaco} who proved that there is a non-separating, incompressible embedding of an orientable surface of any genus (and one or two boundary components) in $H_2$. Jaco then asked whether there exist separating, incompressible embeddings of surfaces with arbitrarily high genus in $H_2$. This question was answered in the affirmative independently by Eudave-Mu\~noz \cite{Eudave}, Howards \cite{Howards} and  Qiu \cite{Qiu 1}.\\                

We extend Jaco's, Qiu's, Howards' and Eudave-Mu\~noz's results by proving that for any given compact surface with boundary there is an incompressible embedding in a handlebody of genus greater than or equal to two. (That is, for a compact surface of any genus and any number of boundary components greater than or equal to one there is such an embedding.) If the surface is orientable we can make this embedding either non-separating or separating. If the surface is non-orientable we can only expect to get a non-separating embedding as the handlebody is orientable. Some examples of the incompressible embeddings that we construct can be seen in figures  \ref{sep_examples.pdf} and \ref{non_or_2_1.pdf}.

\begin{figure}[htb]
\centering
\includegraphics[width=\textwidth]{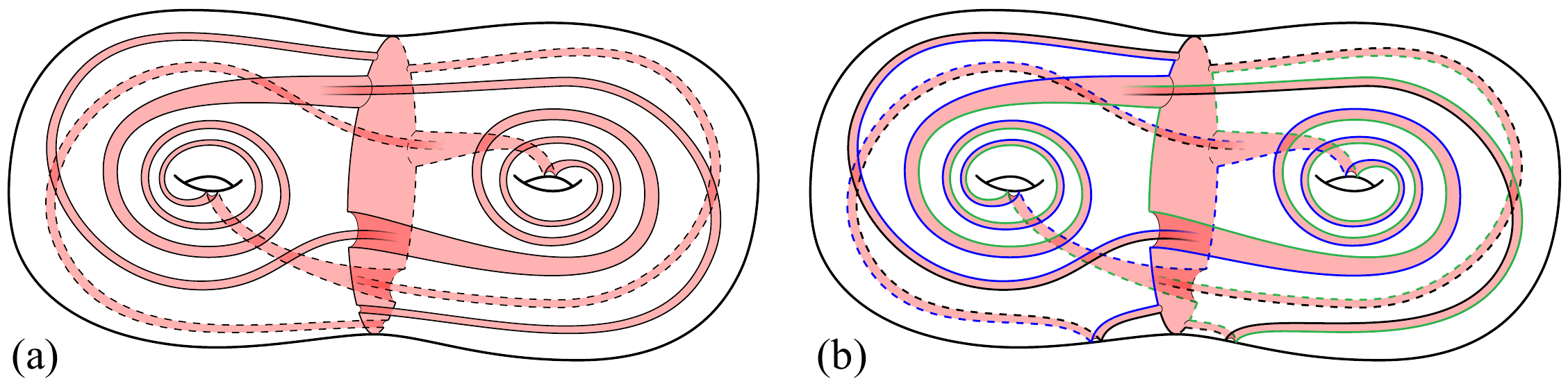}
\caption{Separating incompressible embeddings in $H_2$ of (a) a surface with genus two and one boundary component and (b) a torus with three boundary components. For clarity, when one ``band'' of the surface is nested within another we only draw the inner band as it enters and exits the tube formed between the outer band and $\bdry H_2$.} 
\label{sep_examples.pdf}
\end{figure}

\begin{figure}[htb]
\centering
\includegraphics[width=0.5\textwidth]{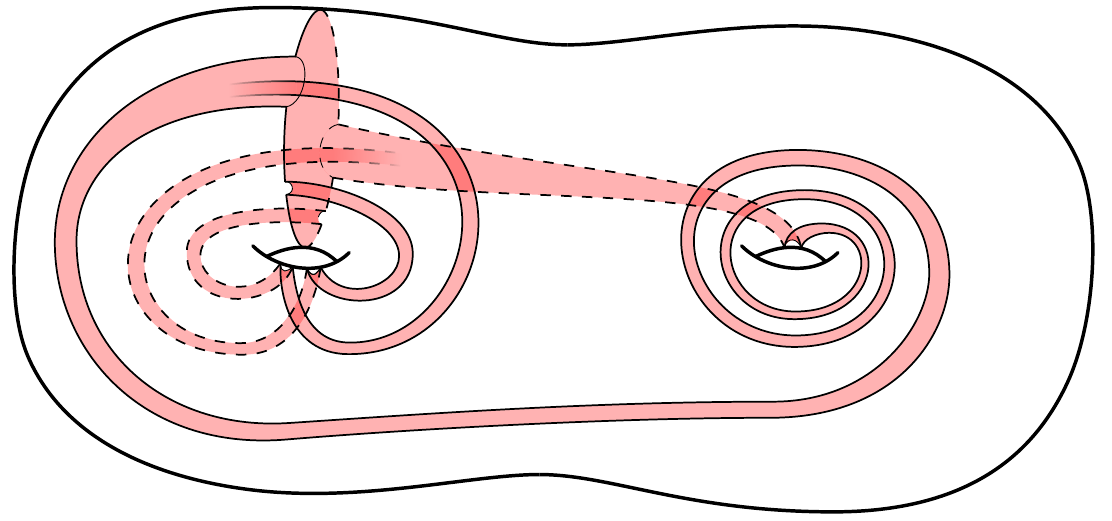}
\caption{Klein bottle with one boundary component incompressibly embedded in $H_2$.} 
\label{non_or_2_1.pdf}
\end{figure}

We make a further extension in section \ref{b-reducible manifold}, proving that we have similar incompressible embeddings in a 3-manifold with a compressible boundary component of genus two or greater. With the same conditions on the 3-manifold, unpublished work of Qiu \cite{Qiu 2} proves the existence of separating, incompressible embeddings of compact surfaces with any genus (and one or two boundary components).\\

In this paper we study the property of a surface having a $\pi_1$-injective embedding, which implies incompressibility.\\

In section \ref{non-separating case} we prove the following result:

\begin{thm}\label{all surfaces in H2}
For each compact surface $S$ with boundary there is a non-separating $\pi_1$-injective embedding of $S$ in $H_2$.
\end{thm}

We prove in section \ref{separating case} that when $S$ is orientable there is also a separating embedding:
 
 \begin{thm}\label{sep orientable in H2}
 For each compact, orientable surface $S$ with boundary there is a separating $\pi_1$-injective embedding of $S$ in $H_2$. 
 \end{thm}
 
 In section \ref{b-reducible manifold}, we extend these results to embeddings in an arbitrary boundary reducible $3$-manifold $M$ with a compressible boundary component of genus greater than or equal to two. We say that such a boundary component $B$ is of {\bf type \emph{ns}} if there is a compression disk $D$ for $B$ such that $\bdry D$  does not separate $B$. We obtain the following general theorems:
 
\begin{thm}\label{all surfaces in M}
Suppose $M$ has a compressible boundary component of genus greater than or equal to two, of type \emph{ns}. Then
for each compact surface $S$ with boundary there is a non-separating $\pi_1$-injective embedding of $S$ in $M$. Furthermore, if $S$ is orientable there is also a separating $\pi_1$-injective embedding of $S$ in $M$.
\end{thm}
 
 \begin{thm}\label{sep orientable in M}
Suppose $M$ has a compressible boundary component, $B$,  of genus greater than or equal to two. Then for each compact, orientable surface $S$ with boundary there is a $\pi_1$-injective embedding of $S$ in $M$. Furthermore, if $B$ has a (non-)separating compressing disk in $M$ then $S$ has a (resp., non-)separating $\pi_1$-injective embedding in $M$.
 \end{thm}
 
 \begin{rmk}
 Consider a manifold $M$ which is the union of two copies of $T \times [0,1]$ (where $T$ denotes the torus) identified to each other along a disk in each $T \times \{0\}$. It is not hard to prove that $M$ has a compressible boundary component of genus greater than or equal to two, but no such boundary component of type \emph{ns}. Therefore, $M$ is an example of a manifold that satisfies the conditions of theorem \ref{sep orientable in M} but not of theorem \ref{all surfaces in M}.
 \end{rmk}

 In sections \ref{non-separating case} and  \ref{separating case} we use a similar standard technique to the one used by Qiu in \cite{Qiu 1}. We start with either a separating or a non-separating incompressible disk in $H_2$, and by adding certain boundary parallel bands to the disk we construct the surfaces with the desired properties. For the extension to the embedding in a 3-manifold with a compressible boundary component of genus greater than or equal to two, in section \ref{b-reducible manifold}, we first extend to compression bodies with a compressible boundary component of genus greater than or equal to two. Then, for the general case, we use a compression body in $M$ obtained from a compressible boundary component of genus greater than or equal to two.\\  
 
In this paper, for a given topological space $X$,  $|X|$ denotes the number of connected components of $X$.\\

\begin{center}
\textit {Acknowledgment}\\
\end{center}

The authors thank Cameron Gordon for helpful discussions.

\section{The non-separating case}
\label{non-separating case}
\subsection{Orientable surfaces}
\label{non-separating case orientable}  

\begin{figure}[htb]
\centering
\includegraphics[width=0.5\textwidth]{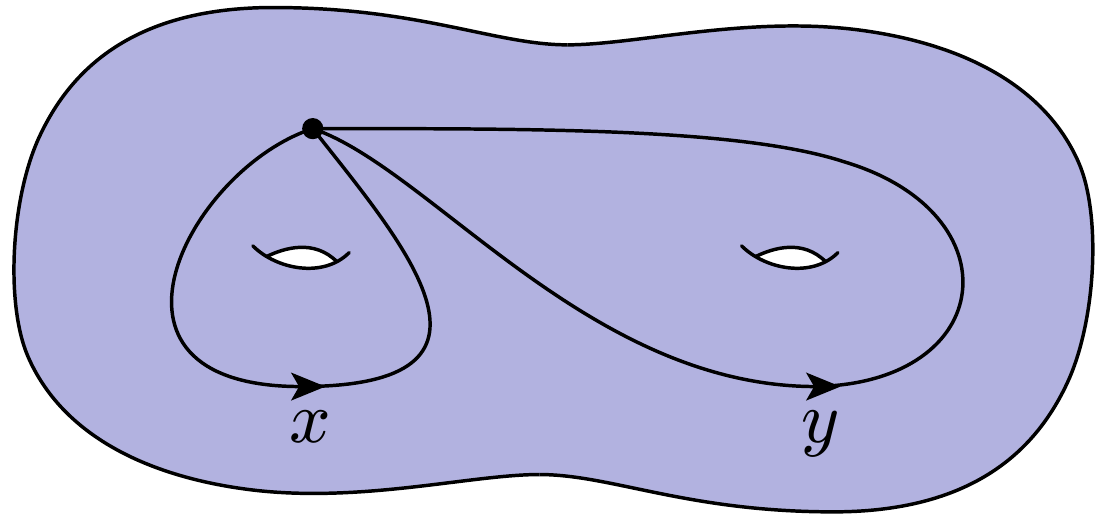}
\caption{Basepoint and generators for $\pi_1(H_2)$.} 
\label{handlebody_non_sep_pi1.pdf}
\end{figure}

The goal of this section is to construct a non-separating, $\pi_1$-injective embedding of each orientable compact surface with boundary, of genus $g$ and number $b\geq 1$ of boundary components, in $H_2$.\\  

First, fix a basepoint and generators for $\pi_1(H_2)$ as in figure \ref{handlebody_non_sep_pi1.pdf}. The surfaces we will construct consist of a single non-separating disk, embedded in the handlebody $H_2$, with various bands connecting parts of the boundary of the disk, contained within a small product neighbourhood of $\bdry H_2$. See figure \ref{handlebody_non_sep_or_surface.pdf}. 
\begin{figure}[htb] 
\centering
\includegraphics[width=\textwidth]{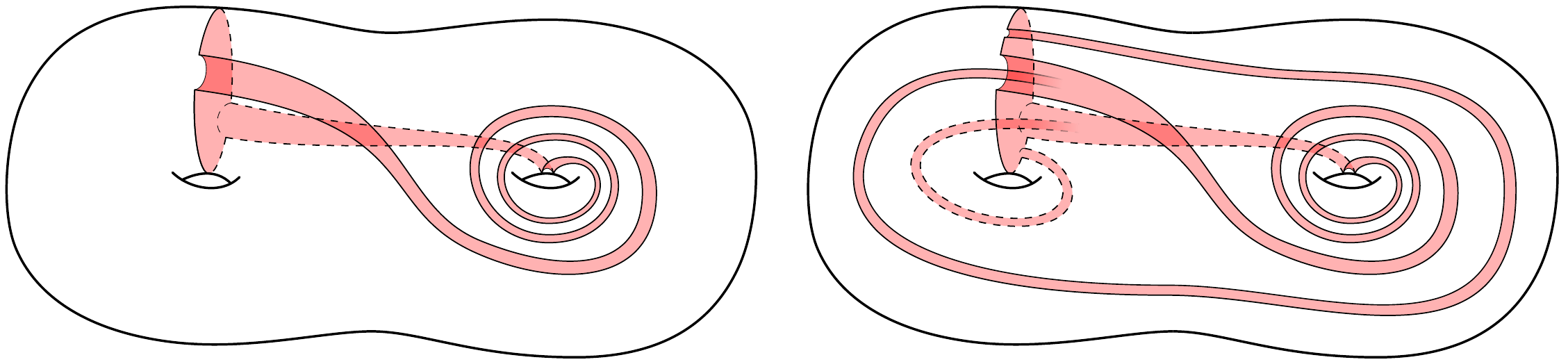}
\caption{Incompressible surfaces embedded in the handlebody of genus $2$, the first two steps. The second band added is nested within the previous one.}
\label{handlebody_non_sep_or_surface.pdf}
\end{figure}
We will add bands in two stages: In the first stage we increase the genus of the surface to the desired value $g$, with the surface having only one boundary component. In the second stage we add bands to increase the number of boundary components to $b$.\\

In the first diagram of figure \ref{handlebody_non_sep_or_surface.pdf} we see the first added band, which we refer to as $\alpha_1$. The $i$th band will be referred to as $\alpha_i$, and we will abuse notation in also referring to the obvious corresponding generator of $\pi_1$ of the surface as $\alpha_i$. In the second diagram of figure \ref{handlebody_non_sep_or_surface.pdf} we see added the band $\alpha_2$, which is nested within $\alpha_1$, in the following manner: The band $\alpha_2$ starts at the boundary of the disk, goes around the handlebody (see figure  \ref{handlebody_non_sep_or_surface.pdf})  and returns to the disk, enters the tunnel formed by $\alpha_1$  and follows it around, then exits, goes around the handlebody again and ends by joining the boundary of the disk. In general $\alpha_i$ and $\alpha_{i-1}$ will be related similarly.\\

In this example, the image of $\alpha_1$ in $\pi_1(H_2)$ is $y^3$ and the image of $\alpha_2$ is $y^{-1}x^{-1}y^3x$.\\

\begin{figure}[htb]
\centering
\includegraphics[width=\textwidth]{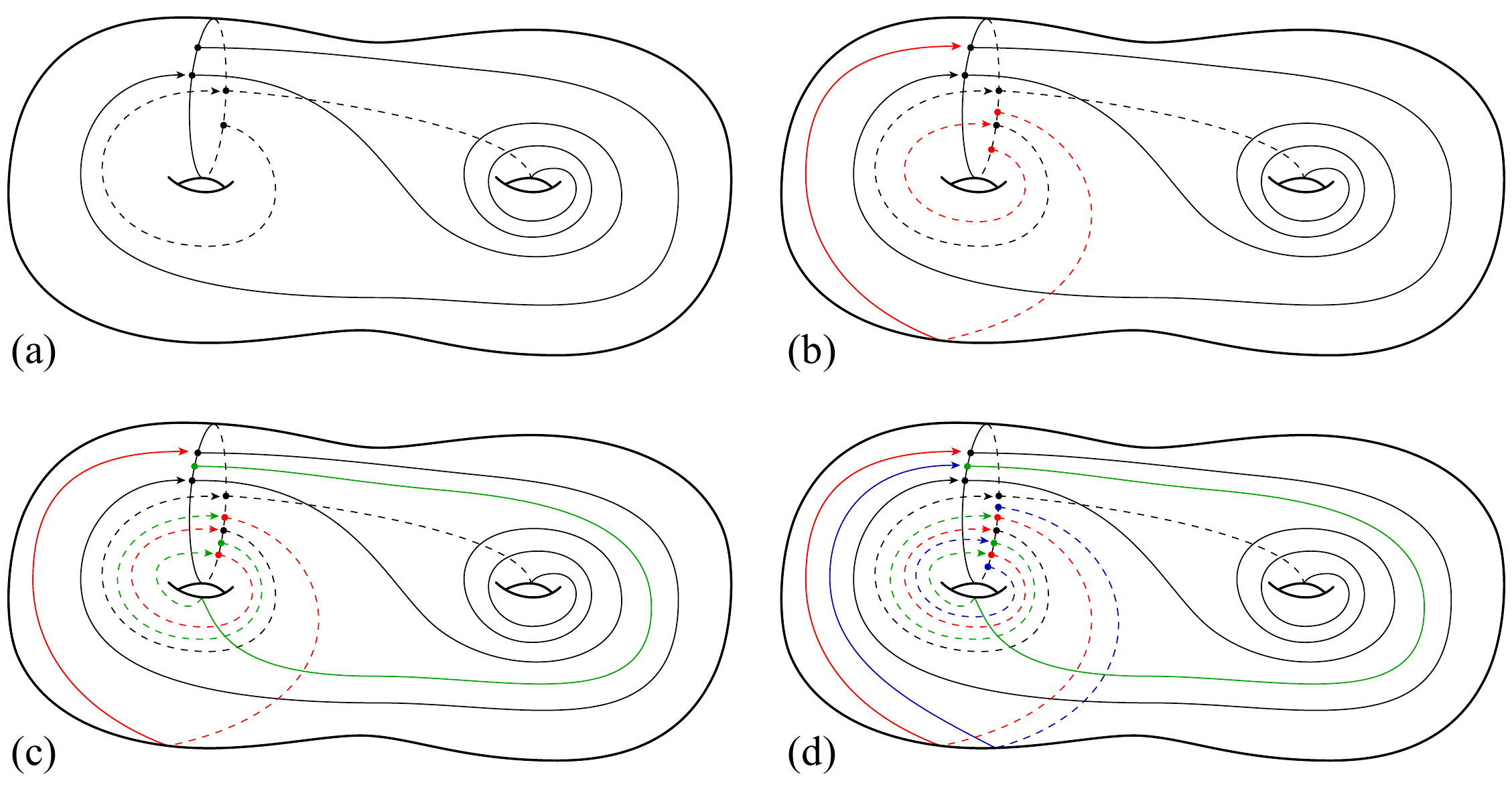}
\caption{Schematic diagrams of non-separating orientable surfaces within $H_2$.}
\label{handlebody_non_sep_or_schematic.pdf}
\end{figure}

We pass to a schematic picture in figure \ref{handlebody_non_sep_or_schematic.pdf}, where in place of a band we draw only its core curve, with a dot where the band joins the disk and an arrow where it travels inside the tunnel formed by the previous band. In figure \ref{handlebody_non_sep_or_schematic.pdf}(a) we see the same situation as in the second diagram of figure \ref{handlebody_non_sep_or_surface.pdf}. Note that there are two parts of $\alpha_2$ outside of the $\alpha_1$ tunnel, one of which corresponds to $y^{-1}x^{-1}$ (on the front side of the handlebody) and the other to $x$ (on the back side). We could, if we chose, continue with building $\alpha_3$ in the same way as $\alpha_2$. That is, one part would go around $y^{-1}x^{-1}$ before entering $\alpha_2$, and on exiting the other part would go around $x$. We could continue this indefinitely, the front side curves growing concentrically and the back side curves shrinking concentrically.\\

Instead however, in figure \ref{handlebody_non_sep_or_schematic.pdf}(b) we make what we call a {\bf changeover} in the front side curve, looping it around to the back side. The back side curve does not changeover yet. In figure \ref{handlebody_non_sep_or_schematic.pdf}(c) we continue what was the front side curve with an $x$ loop on the back side, and have no choice in this because the curves cannot cross if the surface is to be embedded. We do however make a changeover in the back side curve, looping it through the $x$ hole and around the front side of the handlebody. We refer to these two changes in path as a {\bf pair of changeovers}. In figure \ref{handlebody_non_sep_or_schematic.pdf}(d) we continue with yet another band, in this case the front side curve is the first of a new pair of changeovers. \\

\begin{figure}[htbp]
\centering
\includegraphics[width=\textwidth]{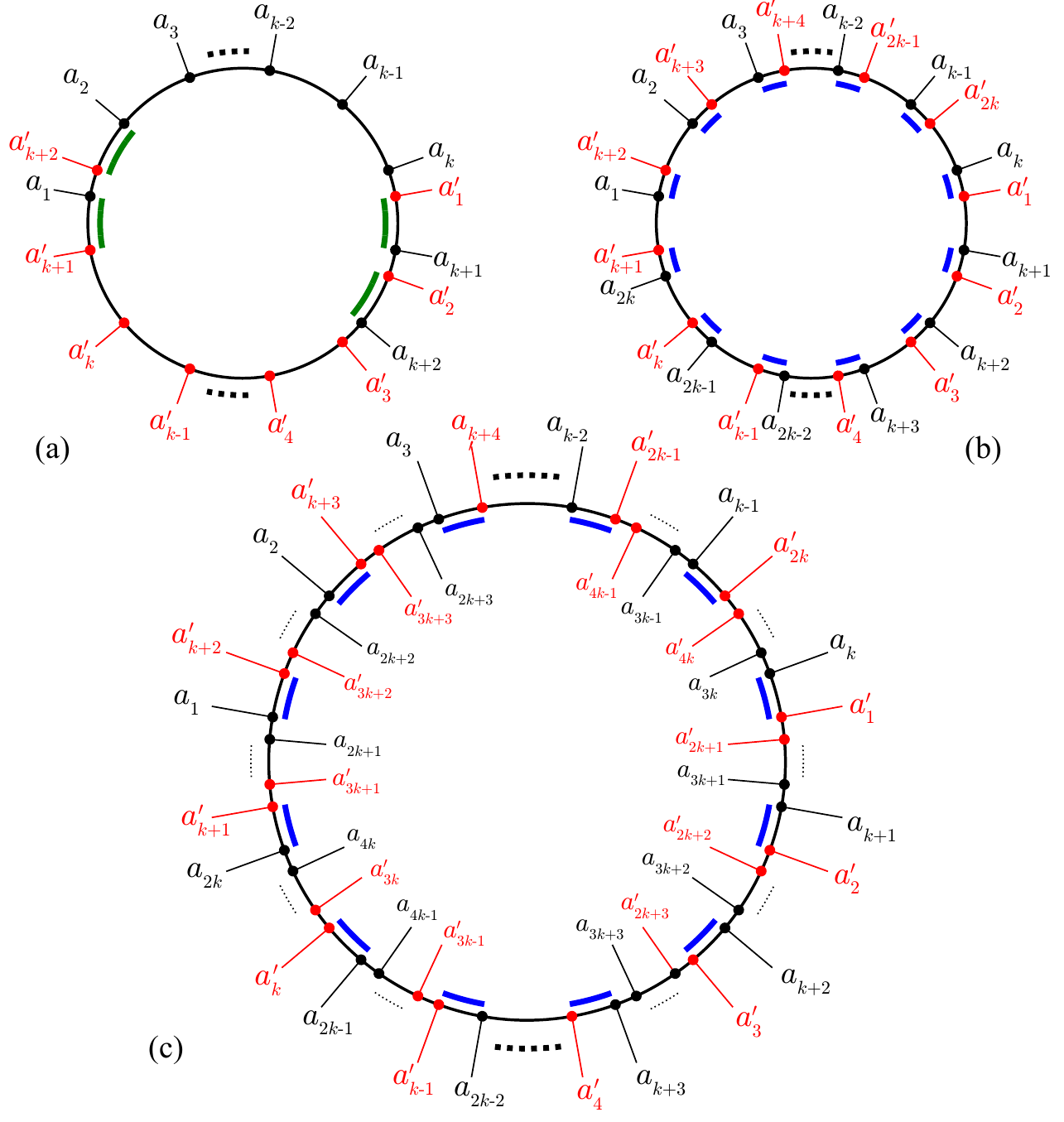}
\caption{Orderings of band ends around the disk. Figure \ref{handlebody_non_sep_or_schematic.pdf}(c) corresponds to (a) in this figure, with $k=2$.}                
\label{handlebody_non_sep_or_disk.pdf}
\end{figure}

We pass once more to a yet more schematic picture in figure \ref{handlebody_non_sep_or_disk.pdf}. Here we show only the pattern of band ends as they meet the disk. We label the ends of $\alpha_k$ as $a_k$ and $a'_k$, where $a_1$ is the end on the front side of the handlebody, $a'_1$ on the back, and in following $\alpha_i$ from $a_i$ ($a'_i$) we enter the $\alpha_{i-1}$ tunnel at $a_{i-1}$ (resp., $a'_{i-1}$).\\

In figure \ref{handlebody_non_sep_or_disk.pdf}(a) we have waited until $\alpha_{k+1}$ (for some $k\geq 2$) to make the first changeover of the first pair of changeovers, and then $\alpha_{k+2}$ for the second changeover. This is the point at which we switch from increasing genus to increasing the number of boundary components, as will be proven later in this section. After a pair of changeovers we don't have a choice of where to add the ends of the next $k-2$ bands in the boundary of the starting disk. On the band after adding these $k-2$ bands, say $\alpha_{jk+1}$ for some positive $j$, we have a choice for the ends of the band. We choose to continue with another pair of changeovers by adding $\alpha_{jk+1}$ as the first changeover and $\alpha_{jk+2}$ as the second. Again, for the next $k-2$ bands we don't have a choice of where to add its ends on the disk. We continue this process until we have added the desired number of bands. The continuation of the pattern obtained by the ends of the bands on the starting disk is as shown in figure \ref{handlebody_non_sep_or_disk.pdf}(b, c).\\

Let $S_{n,k}$ be the surface described above using $n$ bands and starting the first changeover by adding the band $\alpha_{k+1}$, for some $k\geq 2$.

\begin{lemma}\label{adding bands}
Let $S$ be a surface with boundary. Let $S'$ be the surface obtained from $S$ by adding a band $B = [0,1] \times [0,1]$, gluing the two ends $\{0\} \times [0,1]$ and $\{1\} \times [0,1]$ of $B$ to segments of $\bdry S$.
\begin{enumerate}
\item[(a)] If the two ends of $B$ are glued to different components of $\bdry S$ then $|\bdry S'| = |\bdry S| - 1$.

\item[(b)] If the two ends of $B$ are glued to the same component of $\bdry S$ and $B$ has an even number of half twists then $|\bdry S'| = |\bdry S| + 1$.

\item[(c)] If the two ends of $B$ are glued to the same component of $\bdry S$ and $B$ has an odd number of half twists then $|\bdry S'| = |\bdry S|$.\footnote{Note that the parity of the number of half twists is well defined for a band whose ends are attached to the same boundary component of $S$.}
\end{enumerate}
\end{lemma}

\begin{proof}
These statements are not difficult to verify.
\end{proof}

\begin{prop}\label{or sep genus and boundary}
Let $k=2g$, for some natural number $g\geq 1$. Then $S_{n,k}$ has genus $g$ and $n-2g+1$ boundary components.
\end{prop} 
\begin{proof}
Denote by $\Sigma_{s}$ the surface obtained from adding the first $s$ bands in the construction of $S_{n,k}$. All bands added in the construction in this section have an even number of half twists because $\partial H_2$ is orientable and the ends of each band meet the boundary of the disk from the same direction.\\
The surface $\Sigma_0$ is the starting disk so $|\partial \Sigma_0|=1$. Let $i$ be such that $2i< k$.\footnote{So we are considering surfaces obtained in the construction before the first pair of changeovers.} Assume that $|\partial \Sigma_{2i}|=1$.  By Lemma \ref{adding bands}(b) $|\partial \Sigma_{2i+1}|=2$. By construction, the ends of the band $\alpha_{2i+2}$ are on the different boundary components of $\Sigma_{2i+1}$. Therefore, by Lemma \ref{adding bands}(a), $|\partial \Sigma_{2i+2}|=1$. By induction on $i$, $|\partial \Sigma_k|=1$.\\

By the construction of $S_{n,k}$ we can check that the band ends of the bands after the first pair of changeovers, $\alpha_{k+1},\ldots,\alpha_n$, are on the same boundary component. This means, by Lemma \ref{adding bands}(b), that each of these bands added increases the number of boundary components by one. Therefore, the number of boundary components of $S_{n,k}$ is $n-k+1$. The Euler characteristic of $S_{n,k}$ is $1-n$ so we conclude that $S_{n,k}$ has genus $g$.
\end{proof}

We require a special case to obtain a surface with genus zero and arbitrarily many boundary components. We start with the same non-separating disk as before and again denote the bands as $\alpha_i$. See figure \ref{handlebody_non_sep_or_schematic_all_bdry.pdf} for the positions of the first three bands. We continue adding bands as follows: the band $\alpha_i$ goes around $x$ on the front side of the handlebody, goes through the tunnel formed by $\alpha_{i-1}$, goes around $x^{-1}$ on the back of the handlebody and ends at the disk. The back side and front side curves on the handlebody shrink concentrically. We denote the surface obtained by adding $n$ bands following this procedure by $S_{n,0}$.\\

\begin{figure}[htb]
\centering
\includegraphics[width=\textwidth]{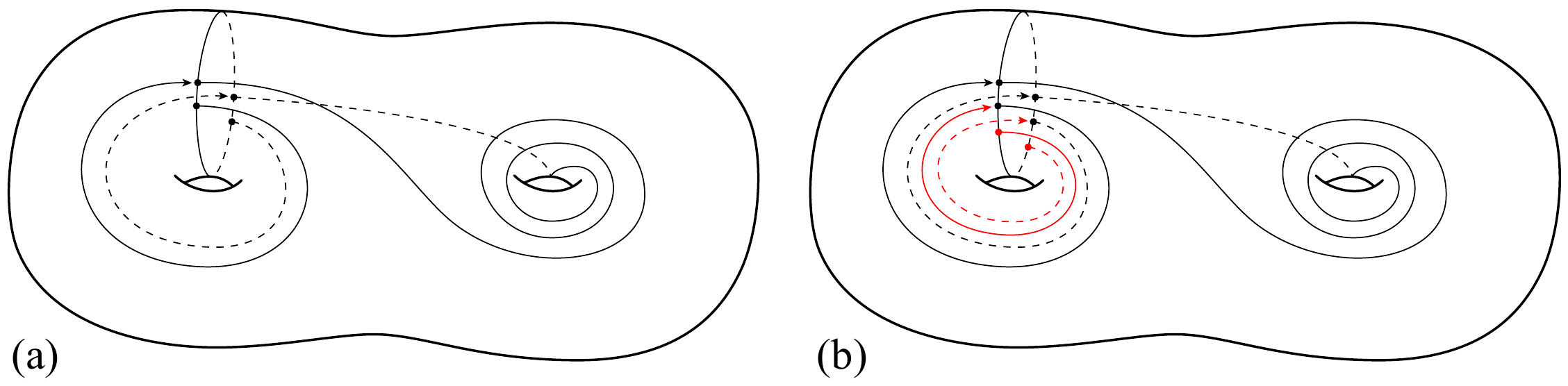}
\caption{Schematic diagrams of non-separating orientable surfaces within $H_2$ with genus zero. The notation is the same as used in figure \ref{handlebody_non_sep_or_schematic.pdf}.}                
\label{handlebody_non_sep_or_schematic_all_bdry.pdf}
\end{figure}

\begin{prop}\label{or sep genus 0 and boundary}
The surface $S_{n,0}$ has genus $0$ and $n+1$ boundary components.
\end{prop} 
\begin{proof}
The argument is similar to the proof of Proposition \ref{or sep genus and boundary}.
\end{proof}

Suppose that $i:S_{n,k}\rightarrow H_2$ is the inclusion map as constructed above and $i_*: \pi_1(S_{n,k})\rightarrow \pi_1(H_2)$ the homomorphism induced by $i$. As in figure \ref{handlebody_non_sep_pi1.pdf}, we denote by $x$ and $y$ the two generators of $\pi_1(H_2)$. We will now prove that the embedded surface in $H_2$ given by $i$ is incompressible by showing that $i_*$ is injective.\\

For $k=0$, from the construction of $S_{n,0}$ we have that 
$$i_*(\alpha_s)=x^{-(s-1)}y^3x^{s-1}\text{, for  }1\leq s\leq n.$$

For $k\geq2$ assume that $n=qk+r$, where $0\leq r<k$.\\ 
Following the construction of the surface, the generators of $\pi_1(S_{n,k})$ are $\alpha_1,\alpha_2,\ldots ,\alpha_n$, and we have:

\begin{align*}
i_*(\alpha_1)&=y^3\\
i_*(\alpha_2)&=y^{-1}x^{-1}y^3x\\
&\vdots\\
i_*(\alpha_k)&=(y^{-1}x^{-1})^{k-1}y^3x^{k-1}\\
i_*(\alpha_{k+1})&=x^{-1}i_*(\alpha_k)x\\
i_*(\alpha_{k+2})&=x^{-1}i_*(\alpha_{k+1})(xy)\\
&\vdots\displaybreak[0]\\
i_*(\alpha_{(2j)k})&=(x^{-k}(y^{-1}x^{-1})^k)^{j-1}x^{-k}(y^{-1}x^{-1})^{k-1}y^3\\
&\quad (x^{k}(xy)^k)^{j-1}x^k(xy)^{k-1}\\
i_*(\alpha_{(2j)k+1})&=(y^{-1}x^{-1})i_*(\alpha_{(2j)k})(xy)\\
i_*(\alpha_{(2j)k+2})&=(y^{-1}x^{-1})i_*(\alpha_{(2j)k+1})x\\
&\vdots\displaybreak[0]\\
i_*(\alpha_{(2j+1)k})&=((y^{-1}x^{-1})^kx^{-k})^{j}(y^{-1}x^{-1})^{k-1}y^3\\
&\quad (x^{k}(xy)^k)^jx^{k-1}\\
i_*(\alpha_{(2j+1)k+1})&=x^{-1}i_*(\alpha_{(2j+1)k})x\\
i_*(\alpha_{(2j+1)k+2})&=x^{-1}i_*(\alpha_{(2j+1)k+1})(xy)\\
&\vdots\displaybreak[0]\\
\intertext{Either $q$ is even and we have:}
&\vdots\\
i_*(\alpha_{qk})&=(x^{-k}(y^{-1}x^{-1})^k)^{\frac{q}{2}-1}x^{-k}(y^{-1}x^{-1})^{k-1}y^3\\
&\quad(x^{k}(xy)^k)^{\frac{q}{2}-1}x^k(xy)^{k-1}\\
i_*(\alpha_{(qk+1})&=(y^{-1}x^{-1})i_*(\alpha_{qk})(xy)\\
i_*(\alpha_{qk+2})&=(y^{-1}x^{-1})i_*(\alpha_{(qk+1})x\\
&\vdots\\
i_*(\alpha_{qk+r})&=(y^{-1}x^{-1})^r(x^{-k}(y^{-1}x^{-1})^k)^{\frac{q}{2}-1}x^{-k}(y^{-1}x^{-1})^{k-1}y^3\\
&\quad(x^{k}(xy)^k)^{\frac{q}{2}-1}x^k(xy)^kx^{r-1}\displaybreak[0]\\
\intertext{Or $q$ is odd and we have:}
&\vdots\\
i_*(\alpha_{qk})&=((y^{-1}x^{-1})^kx^{-k})^{\frac{q-1}{2}}(y^{-1}x^{-1})^{k-1}y^3\\               
&\quad (x^{k}(xy)^k)^{\frac{q-1}{2}}x^{k-1}\\
i_*(\alpha_{qk+1})&=x^{-1}i_*(\alpha_{qk})x\\
i_*(\alpha_{qk+2})&=x^{-1}i_*(\alpha_{qk+1})(xy)\\
&\vdots\\
i_*(\alpha_{qk+r})&=x^{-r}((y^{-1}x^{-1})^kx^{-k})^{\frac{q-1}{2}}(y^{-1}x^{-1})^{k-1}y^3\\
&\quad (x^{k}(xy)^k)^{\frac{q-1}{2}}x^k(xy)^{r-1}.\\
\end{align*}  

For the proof that $i_*$ is injective we use the following Lemma. For a reduced word $w$, we denote by $L(w)$ the length of the word.
\begin{lemma}\label{word lemma}
Let $F_2$ be the free group with two generators, $g$ and $h$, and with identity element  $e$. Consider $$A=\{v_1g^pv_1', v_2g^pv_2',\cdots, v_ng^pv_n'\} \,\text{with $p\geq 3$,}$$ where $v_ig^pv_i'$, for $i=1,\cdots,n$, is a reduced word on $g$ and $h$,  and $v_i$, $v_i'$ are words where $g$, or $g^{-1}$, alternate with powers of $h$. Furthermore, $v_i$ ends with a power of $h$ and $v_i'$ starts with a power of $h$.  Assume also that $L(v_{i+1})>L(v_i)$ and $L(v_{i+1}')>L(v_i')$, for $i=1,2,\cdots,n-1$. Take $w$ to be a reduced word on the elements of $A$. That is, $w$ is expressed as a product of elements of $A$ or inverses of those elements, with no consecutive cancelling terms. Then $w\neq e$. 
\end{lemma}
\begin{proof}
Let $w=w_1w_2\cdots w_s$ where either $w_i \in A$ or $w_i^{-1} \in A$ and $s\in \mathbb{N}$.
If $s=1$ then $w=w_1$ and $w\neq e$. Let $s\geq 2$ and assume $w=e$. Then as $w\in F_2$ there are two consecutive letters in $w$ that are inverses of each other. However $w_i$ is a reduced word, which means that two consecutive canceling letters cannot be part of the same $w_i$. So, cancelling pairs in $w$ are defined by letters of different $w_i$'s. As either $w_i \in A$ or $w_i^{-1} \in A$ we have $$w_i=u_ig^{\epsilon_i p}u_i', i\in \{1, \ldots, s\}, \epsilon_i \in \{-1,1\},$$ where $u_i$ or $u_i^{-1}$, and $u_i'$ or $u_i'^{-1}$ are elements of $\{v_j, v_j': j=1,\ldots,n\}$.\\   Take $c_i=u_i'u_{i+1}$ from which we can write $w_iw_{i+1}=u_ig^{\epsilon_i p}c_ig^{\epsilon_{i+1} p}u_{i+1}'$.\\  If $c_i=e$ then $L(u_i')=L(u_{i+1})$. If the signs of $\epsilon_i$ and  $\epsilon_{i+1}$ are opposite then it follows that $w_i=w_{i+1}^{-1}$ which contradicts the assumption on $w$. Therefore, when $c_i=e$ we have $\epsilon_i$ equal to $\epsilon_{i+1}$, so the powers $g^{\epsilon_ip}$ and $g^{\epsilon_{i+1}p}$ don't cancel each other. We have, 

 $\begin{array}{lll}
 w&=&u_1g^{\epsilon_1 p}u_1'u_2g^{\epsilon_2 p}u_2'\cdots u_ig^{\epsilon_i p}u_i'u_{i+1}g^{\epsilon_{i+1} p}u_{i+1}'\cdots u_sg^{\epsilon_s p}u_s'\\
 &=& u_1g^{\epsilon_1 p}c_1g^{\epsilon_2 p}c_2\cdots c_{i-1}g^{\epsilon_i p}c_ig^{\epsilon_{i+1} p}c_{i+1}\cdots c_{s-1}g^{\epsilon_s p}u_s'.
 \end{array}$\\
 Canceling each $c_i$ as much as we can, for $i=\{1,\ldots, s-1\}$, either
 \begin{itemize}
 \item[a)] $c_i=e$ , when $u_i$ and $u_i'$ cancel each other, and, as above, $\epsilon_i=\epsilon_{i+1}$ which means the respective powers of $g$ aren't canceled; or\\
 \item[b)] $c_i=h^{t_1}\cdots h^{t_2} g^{\pm 1}$, which can only happen if part of $u_i'$ cancels the whole of $u_{i+1}$; or\\
 \item[c)] $c_i=g^{\pm 1}h^{t_1}\cdots h^{t_2}$, which can only happen if part of $u_{i+1}$ cancels the whole of $u_i'$;  or\\
 \item[d)] $c_i=h^{t_1}\cdots h^{t_2}$, which can only happen if part of $u_i'$ (or $u_{i+1}$) cancels the whole of $u_{i+1}$ (resp., $u_i'$) or when none of $u_i'$ and $u_{i+1}$ is canceled;
 \end{itemize}
 where $t_1$ and $t_2$ represent some non-zero exponent.\\
 As $p\geq 3$, the exponent $\epsilon_i  p$ of $g$ from the word $w_i$, with $i\in\{1,2,\ldots,s\}$, in $w$ is never entirely canceled. Then $L(w)\geq s$. Hence $w\neq e$.               
  \end{proof}

 \begin{prop}\label{orientable non-separating}
  The homomorphism $i_*:\pi_1(S_{n,k})\rightarrow \pi_1(H_2)$ is injective.
 \end{prop}
 \begin{proof}
 Take $A=\{i_*(\alpha_1),i_*(\alpha_2),\ldots,i_*(\alpha_n)\}$ and, using the words from above, write $i_*(\alpha_i)=v_iy^3v_i'$, for $i=1,\ldots,n$. We have that the $v_iy^3v_i' $, for $i=1,\ldots,n$,  are reduced words on $x$ and $y$, and $v_i$, $v_i'$ are words where $y$, or its inverse, alternates with powers of $x$. Furthermore, $v_i$ ends with a power of $x$ and $v_i'$ starts with a power of $x$. Also $L(v_{i+1})>L(v_i)$ and $L(v_{i+1}')>L(v_i')$, for $i=1,2,\ldots,n-1$. So, we are under the hypotheses of Lemma \ref{word lemma}.\\

Let $w=i_*(\alpha)$ for some $\alpha\in \pi_1(S_{n,k})$ with $\alpha\neq e$. We can write $W$ as a reduced word on the elements of $A$: $w=w_1w_2\cdots w_s$, where  either $w_i \in A$ or $w_i^{-1} \in A$. Then by Lemma \ref{word lemma}, $w\neq e$. Hence $i_*$ is injective. \end{proof}
 
The above results prove Theorem \ref{all surfaces in H2} for orientable surfaces.

 \subsection{Non-orientable surfaces}

The setup for the non-orientable case is similar to the orientable case. We use the same generators for $\pi_1(H_2)$ as in figure \ref{handlebody_non_sep_pi1.pdf}. We again start with a single disk, as in figure \ref{handlebody_non_sep_non_or_schematic_start.pdf}. This time the first $k$ bands will meet the disk from opposite directions, and so have an odd number of half twists, and the constructed surfaces are non-orientable. \\

\begin{figure}[htb]
\centering
\includegraphics[width=\textwidth]{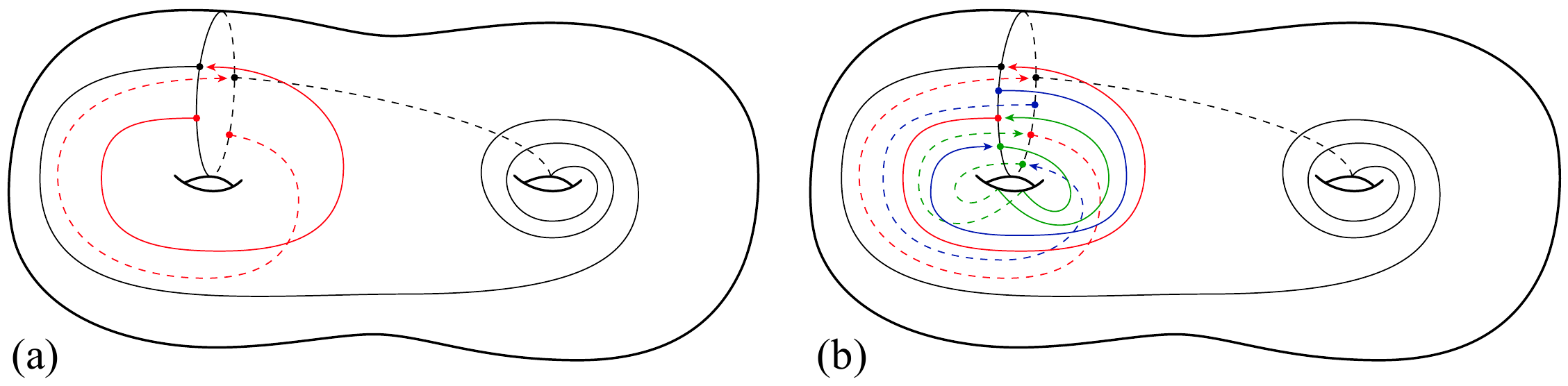}
\caption{Schematic diagrams of non-orientable surfaces within the handlebody. These diagrams show the first stage, which adds non-orientable genus.}
\label{handlebody_non_sep_non_or_schematic_start.pdf}
\end{figure}

\begin{figure}[htbp]
\centering
\includegraphics[width=\textwidth]{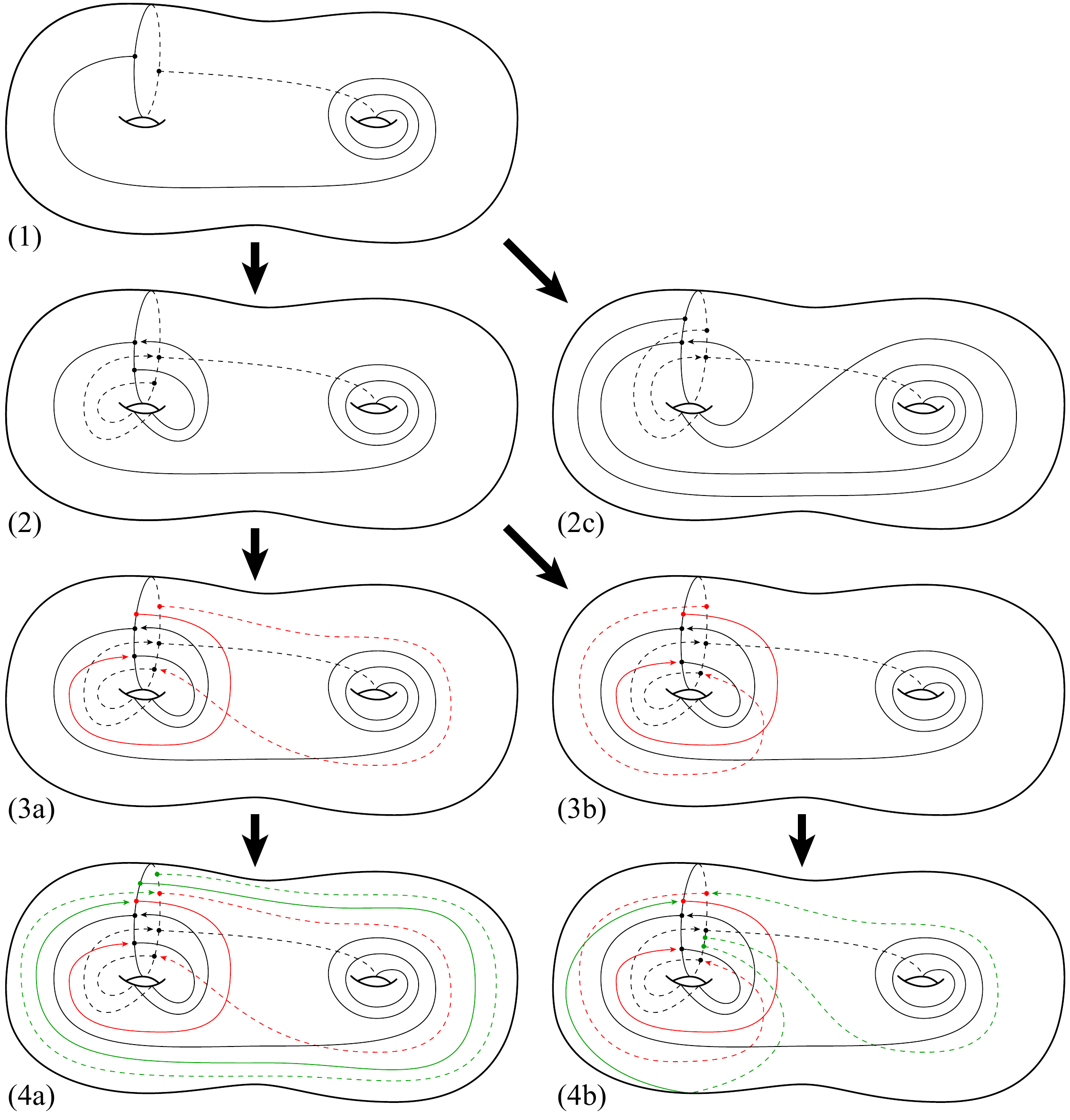}
\caption{Schematic diagrams of non-orientable surfaces within the handlebody. This shows the start of the second stage, for even $k$ in (3a,4a), for $k=1$ in (2c) and for other odd $k$ in (4b).}
\label{handlebody_non_sep_non_or_schematic_continue.pdf}
\end{figure}

As before we refer to the bands as $\alpha_i$ with ends $a_i, a'_i$. Also, as in the orientable case there are two stages, and again the first stage (adding the first $k$ bands) produces (non-orientable) genus and the second stage produces boundary components.\\

We show an example of the first stage in figure \ref{handlebody_non_sep_non_or_schematic_start.pdf}. The band $\alpha_1$ is similar to as in the orientable case, except that one end meets the disk from the other side, and so the image of it in $\pi_1(H_2)$ is $xy^3$. The band $\alpha_2$ spirals into the $x$ handle from both sides. We continue spiraling in, but we choose some point at which the two ends of the band cross over their positions at which they meet the disk. We continue adding bands, spiraling outwards again, until the ends $a_i$ and $a'_i$ are adjacent and below $a_1$ and $a'_1$ in the figure. This happens at $\alpha_4$ in figure \ref{handlebody_non_sep_non_or_schematic_start.pdf}, and in general will happen at an even subscript.\\

If we want the number of bands in the first stage, $k$, to be even, then we end the first stage here. If not then we add one further band continuing the spiral outwards around the $x$ handle, so that the ends of the band are adjacent and above $a_1$ and $a'_1$. See figure \ref{handlebody_non_sep_non_or_schematic_continue.pdf}(3b). In figure \ref{handlebody_non_sep_non_or_schematic_continue.pdf} the two ends of the band cross over earlier than in figure \ref{handlebody_non_sep_non_or_schematic_start.pdf}, and we see the start of the second stage for $k$ equal to 1, 2 and 3.\\

The second stage is slightly different for the two cases: $k$ even and $k$ odd. We also have a different second stage for $k=1$. Figure \ref{handlebody_non_sep_non_or_schematic_continue.pdf}(3a,4a) shows the continuation for even $k$. In each of these diagrams we add a band whose ends are adjacent to each other, and as they meet the disk from the same side they have an even number of half twists. The pattern continues with further bands that add parts that continue moving concentrically outwards around both handles of the handlebody. Each added band has adjacent ends that meet the disk from the same side.\\

Figure \ref{handlebody_non_sep_non_or_schematic_continue.pdf}(4b) shows the continuation for odd $k \geq 3$. Again we add a band whose ends are adjacent to each other and meet the disk from the same side, and so have an even number of half twists. Once again the pattern continues on with new bands with adjacent ends that meet the disk from the same side. This time it is a little less obvious where these bands will go, but it isn't hard to see that we can always continue the bands in a manner that spirals around the $x$ handle in the anti-clockwise direction.\\

Figure \ref{handlebody_non_sep_non_or_schematic_continue.pdf}(2c) shows the continuation for $k=1$. Similarly to the case for even $k$, the pattern continues with adding parts that move concentrically outwards around both handles of the handlebody.\\

Let $S_{n,k}$ be the surface described above using $n$ bands and ending the first stage at step $k\geq1$.\\

\begin{prop}
The surface $S_{n,k}$ has non-orientable genus $k$ and $n-k+1$ boundary components.
\end{prop}
\begin{proof}
Denote by $\Sigma_s$ the surface obtained from adding the first $s$ bands in the construction of $S_{n,k}$. In the construction of $S_{n,k}$ the ends of the first $k$ bands meet the boundary of the disk from opposite directions and the ends of the remaining $n-k$ bands meet the boundary of the disk from the same direction. As $\partial H_2$ is orientable, the first $k$ bands have an odd number of half twists and the next $n-k$ bands have an even number of half twists.\\

The starting disk is $\Sigma_0$ so $|\partial \Sigma_0|=1$. Let $i$ be such that $i<k$. Assume that $|\partial\Sigma_i|=1$. As $i+1\leq k$, the band $\alpha_{i+1}$ has an odd number of half twists. By Lemma \ref{adding bands} (c) we have that $|\partial \Sigma_{i+1}|=1$. So, by induction on $i$ we have that $|\partial \Sigma_k|=1$.\\

By the construction of $S_{n,k}$ we can check that the band ends of $\alpha_{k+1},\ldots,\alpha_n$ are on the same boundary component. Also, these bands have an even number of half twists. So, from Lemma \ref{adding bands} (b) each of these bands added increases the number of boundary components by one. Therefore the number of boundary components of $S_{n,k}$ is $n-k+1$. As the Euler characteristic of $S_{n,k}$ is $1-n$ the surface $S_{n,k}$ has non-orientable genus $k$.
\end{proof}

Suppose that $i:S_{n,k}\rightarrow H_2$ is the inclusion map as constructed above and $i_*: \pi_1(S_{n,k})\rightarrow \pi_1(H_2)$ the induced homomorphism by $i$. As mentioned before, we denote by $x$ and $y$ the two generators of $\pi_1(H_2)$. \\ 

Following the construction of the surface, the generators of $\pi_1(S_{n,k})$ are $\alpha_1,\,\alpha_2,\,\cdots ,\, \alpha_n$.\\

We have:
\begin{align*}
i_*(\alpha_1)&=xy^3\\
i_*(\alpha_2)&=x(xy^3)x\\
&\vdots\\
i_*(\alpha_k)&=x^{k-1}(xy^3)x^{k-1}\\
&\vdots\displaybreak[0]\\
\intertext{Either $k$ is even and we have:}
&\vdots\\
i_*(\alpha_{k+1})&=y^{-1}x^{k-1}(xy^3)x^{k-1}x\\
i_*(\alpha_{k+2})&=(y^{-1}x^{-1})y^{-1}x^ky^3x^k(xy)\\
&\vdots\\
i_*(\alpha_{k+(n-k)})&=(y^{-1}x^{-1})^{n-k-1}y^{-1}x^ky^3x^k(xy)^{n-k-1} \displaybreak[0]\\
\intertext{Or $k$ is odd and we consider the two cases:}
\intertext{For $k=1$ we have:}
&\vdots\\
i_*(\alpha_{2})&=x(xy^3)xy^{-1}x^{-1}\\
i_*(\alpha_{3})&=(xy)x^2y^3xy^{-1}x^{-1}(y^{-1}x^{-1})\\
i_*(\alpha_{4})&=(xy)^2x^2y^3x(y^{-1}x^{-1})^3\\
&\vdots\\
i_*(\alpha_{n})&=(xy)^{n-2}x^2y^3x(y^{-1}x^{-1})^{n-1}\displaybreak[0]\\
\intertext{For $k\geq3$ we have:}
&\vdots\\
i_*(\alpha_{k+1})&=yx^{k-1}(xy^3)x^{k-1}x\\
i_*(\alpha_{k+2})&=x^{-1}yx^ky^3x^kx\\
i_*(\alpha_{k+3})&=x^{-2}yx^ky^3x^kx^2\\
&\vdots\\
i_*(\alpha_{k+(n-k)})&=x^{-(n-k-1)}yx^ky^3x^kx^{(n-k-1)}\\
\end{align*}

 \begin{prop}\label{non-orientable non-separating}
 The homomorphism $i_*:\pi_1(S_{n,k})\rightarrow \pi_1(H_2)$ is injective.
 \end{prop}
 \begin{proof}
 Take $A=\{i_*(\alpha_1),i_*(\alpha_2),\cdots,i_*(\alpha_n)\}$ and, using the words from above, write $i_*(\alpha_i)=v_iy^3v_i'$, for $i=1,\ldots,n$. We have that the $v_iy^3v_i' $, for $i=1,\ldots,n$,  are reduced words on $x$ and $y$, and $v_i$, $v_i'$ are words where $y$, or its inverse, alternates with powers of $x$. Furthermore, $v_i$ ends with a power of $x$ and $v_i'$ starts with a power of $x$. Also $L(v_{i+1})>L(v_i)$ and $L(v_{i+1}')>L(v_i')$, for $i=1,2,\ldots,n-1$. So,  we are under the hypothesis of Lemma \ref{word lemma}.\\

Let $w=i_*(\alpha)$ for some $\alpha\in \pi_1(S_{n,k})$ with $\alpha\neq e$. We can write $W$ as a reduced word on the elements of $A$: $w=w_1w_2\cdots w_s$, where either $w_i \in A$ or $w_i^{-1} \in A$. Then by Lemma \ref{word lemma}, $w\neq e$. Hence $i_*$ is injective. \end{proof}
 
From the constructions and results in section \ref{non-separating case}, in particular Propositions \ref{orientable non-separating} and \ref{non-orientable non-separating}, we obtain Theorem \ref{all surfaces in H2}.

\section{The separating case}
\label{separating case}

In this section we construct a separating $\pi_1$-injective embedding of each compact, orientable surface with boundary, of genus $g$ and number  $b\geq 1$ of boundary components, in $H_2$. A separating surface in an orientable manifold must necessarily be orientable, so we consider only this case. We fix a basepoint and generators as in figure \ref{handlebody_sep_pi1.pdf}, which differs very slightly from the non-separating case.\\

\begin{figure}[htb]
\centering
\includegraphics[width=0.5\textwidth]{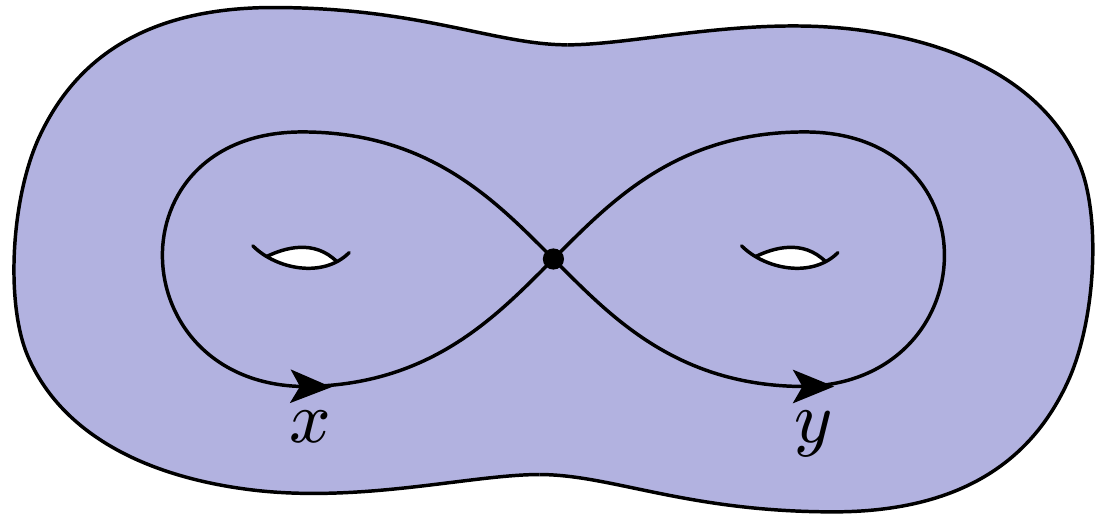}
\caption{Basepoint and generators for $\pi_1(H_2)$, separating case.}
\label{handlebody_sep_pi1.pdf}
\end{figure}

\begin{figure}[htb]
\centering 
\includegraphics[width=\textwidth]{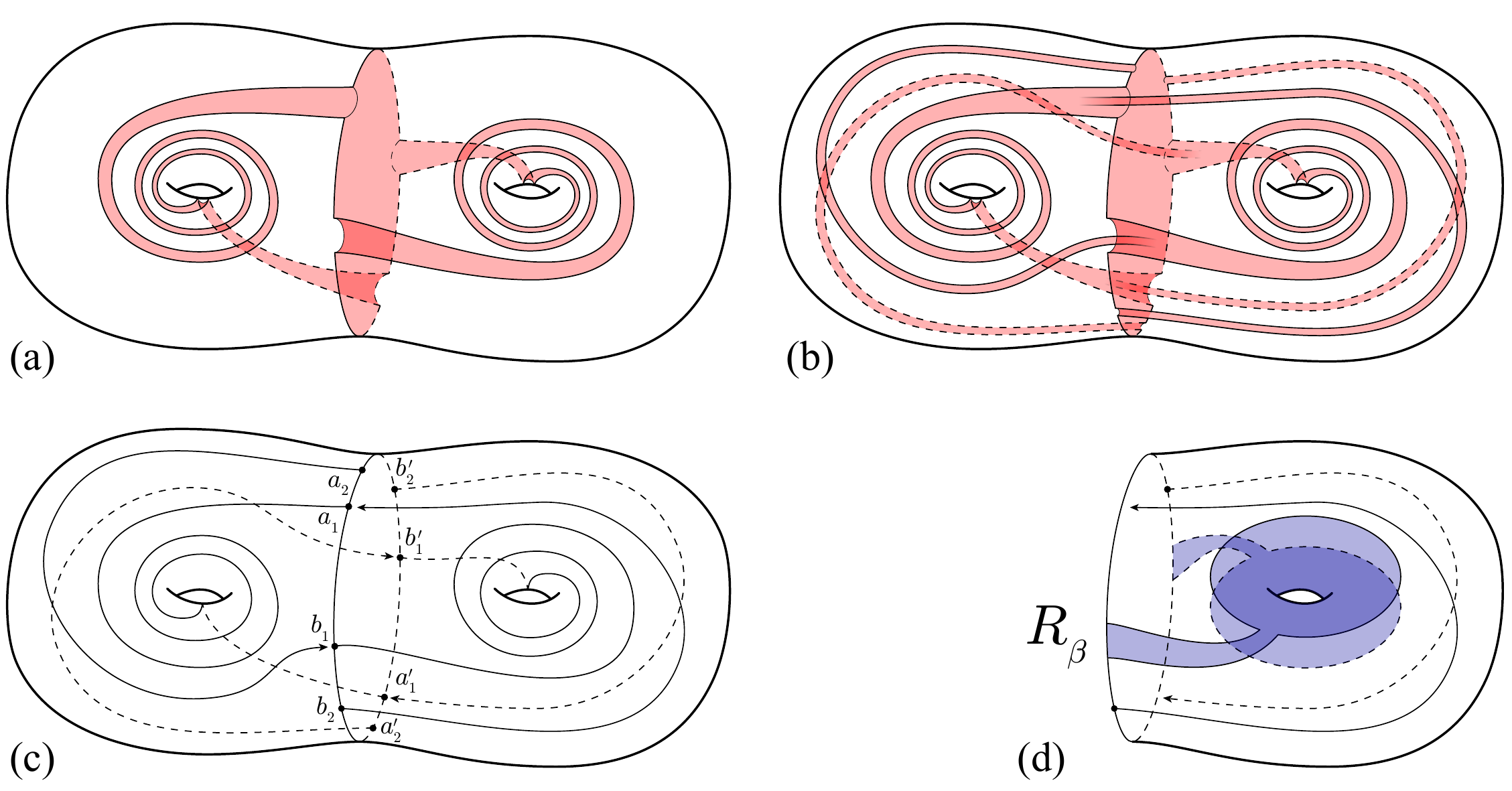}
\caption{Diagrams for the construction of separating surfaces.}
\label{handlebody_sep.pdf}
\end{figure}

\begin{figure}[htb]
\centering
\includegraphics[width=\textwidth]{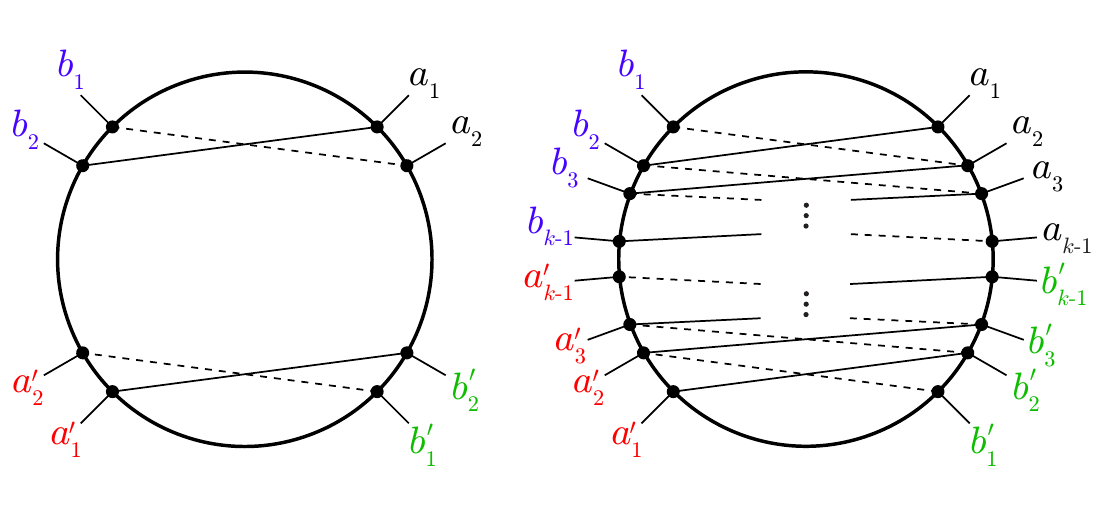}
\caption{On the left, band ends around the disk at the second step. The solid lines between $a_1$ and $b_2$, and between $a'_1$ and $b'_2$ correspond to the curves on the right side of figure \ref{handlebody_sep.pdf}(c). The dashed lines correspond to curves on the left side of that figure. On the right, we can continue this pattern as long as we want before the `changeover'.}
\label{handlebody_sep_disk_step2.pdf}
\end{figure}

\begin{figure}[htbp]
\centering
\includegraphics[width=\textwidth]{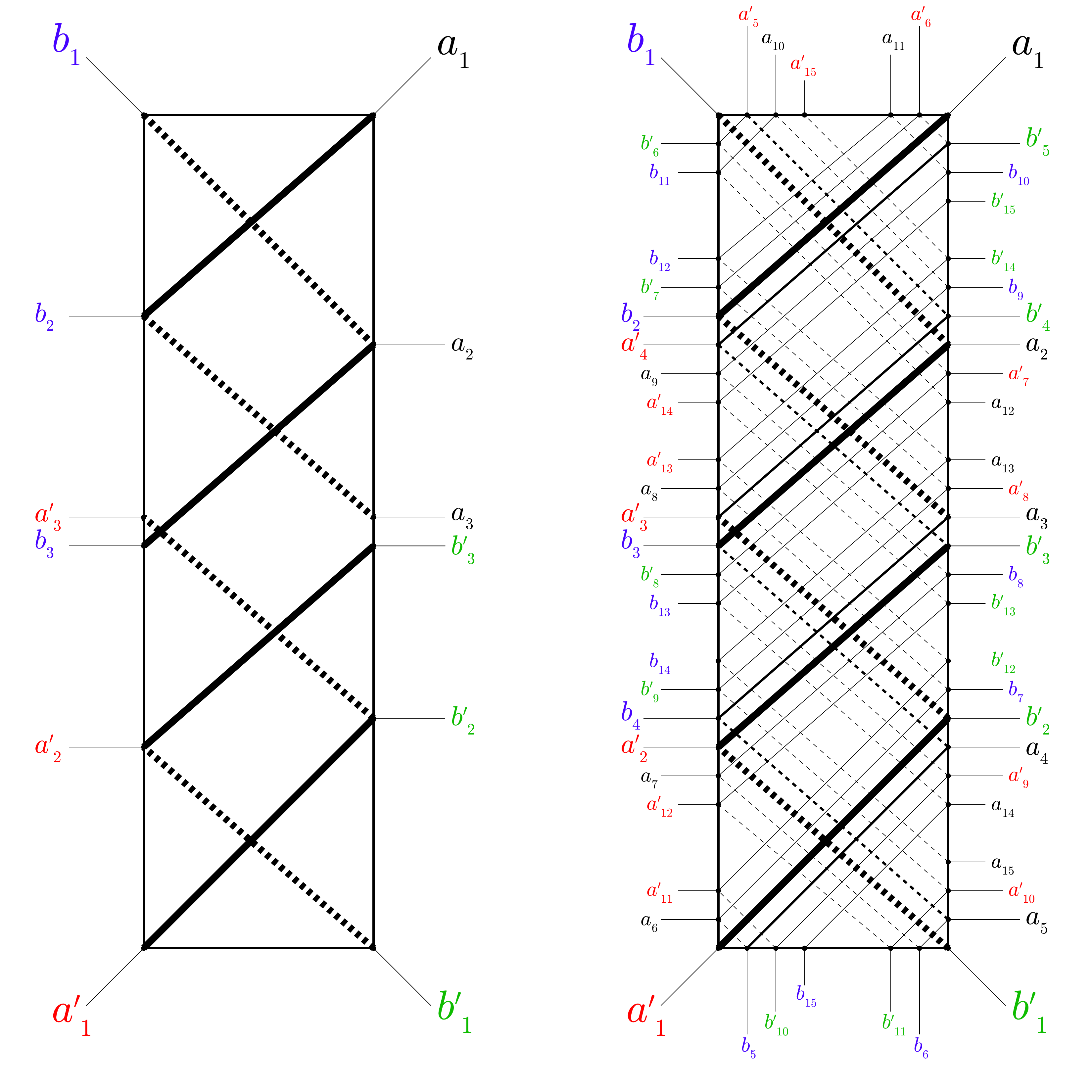}
\caption{Band ends around the disk at further steps. On the left, we continue the left diagram of figure \ref{handlebody_sep_disk_step2.pdf} with the changeover at step 3. On the right, we continue the left diagram by adding bands after the changeover. The lines correspond to curves on the boundary of $H_2$ as in figure \ref{handlebody_sep_disk_step2.pdf}.}
\label{sep_sphere_diagram_example2.pdf}
\end{figure}

Once again the surfaces consist of a single disk with various bands attached to its border contained within a small product neighbourhood of $\bdry H_2$. See figure \ref{handlebody_sep.pdf}(a,b). We construct each surface in a number of steps. Each step consists of adding two bands to the border of the disk. We refer to the bands added in step $i$ as $\alpha_i$ and $\beta_i$. In figure \ref{handlebody_sep.pdf}(a) band $\alpha_1$ is to the left and $\beta_1$ to the right.\\

In figure \ref{handlebody_sep.pdf}(b) we have the second step. Similarly to in the non-separating cases, new bands start at the disk, travel around the boundary of the handlebody until they return to the disk, travel through the tunnel formed by a previous band, re-emerge at the other end of the tunnel, travel around the boundary of the handlebody and finally meet the disk. In step $i+1$, $\alpha_{i+1}$ travels through the tunnel formed by $\beta_i$ and $\beta_{i+1}$ travels through the tunnel formed by $\alpha_i$. This means that the $\alpha_i$ bands always meet the disk approaching from the left, and the $\beta_i$ bands from the right.\\

In figure \ref{handlebody_sep.pdf}(c) we see the second step in schematic form, with the end points of the bands labelled. The band $\alpha_i$ has endpoints $a_i$ and $a'_i$ and $\beta_i$ has endpoints $b_i$ and $b'_i$. The end $a_1$ of $\alpha_1$ is on the front side of the handlebody as we view it, $a'_1$ is on the back side and similarly for $b_1$ and $b'_1$. For further bands, similarly to in the non-separating cases, we choose the label $a_{i+1}$ to be the end of $\alpha_{i+1}$ nearer to the end of $\beta_i$ labelled $b_i$. Here ``nearer'' means in traveling along the band $\alpha_{i+1}$ towards the tunnel formed by $\beta_i$. Similarly for the labelling of $b_{i+1}$ and $b'_{i+1}$.\\

As we add further bands, we will never add new bands that pass through the subset of the right side of the boundary of the handlebody, $R_\beta$, shown shaded in figure \ref{handlebody_sep.pdf}(d). The region $R_\beta$ contains the core curve of $\beta_1$. The complementary region of $R_\beta$ in the half of the boundary of the handlebody shown in figure \ref{handlebody_sep.pdf}(d) is a disk. This disk has boundary containing all of the boundary of the starting disk of our surfaces apart from a small neighbourhood of the endpoints of $\beta_1$. Thus, we can encode the data of the curves on the right side of figure \ref{handlebody_sep.pdf}(c) as curves drawn on a disk with the same boundary as the starting disk of our surfaces. We represent this disk in figure \ref{handlebody_sep_disk_step2.pdf}. The solid curves we draw on this disk will always cross the line between $b_1$ and $b'_1$ (not shown in the figure), as they have to travel around the handle. The curves cannot cross each other, for otherwise the corresponding surface would have self-intersections.\\  

There is a similar region $R_\alpha$ (not shown) on the left side, containing $\alpha_1$, and the above observations all go through similarly. We get a second disk for curves on the left side of the handlebody, and we draw them on the same picture in figure \ref{handlebody_sep_disk_step2.pdf} as dashed lines. In the right diagram of figure \ref{handlebody_sep_disk_step2.pdf} we continue adding bands, and can add as many as we want following this pattern: adding ends moving from the top of the disk down and from the bottom of the disk up but not yet meeting in the middle.\\ 

\begin{figure}[htbp]
\centering
\includegraphics[width=\textwidth]{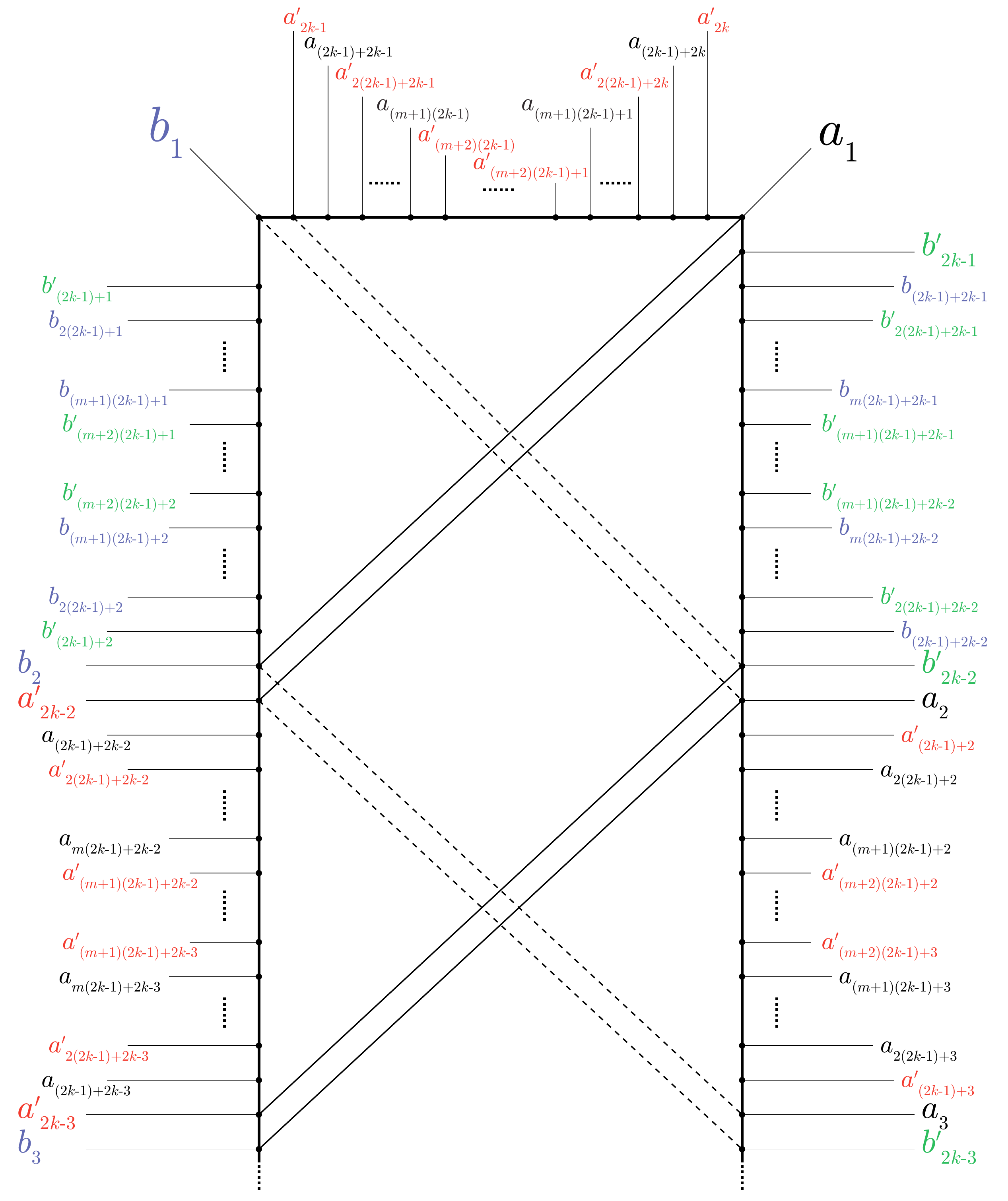}
\caption{The picture at the top of the disk for odd $k$. }
\label{sep_sphere_k_odd_top.pdf}
\end{figure}

\begin{figure}[htbp]
\centering
\includegraphics[width=\textwidth]{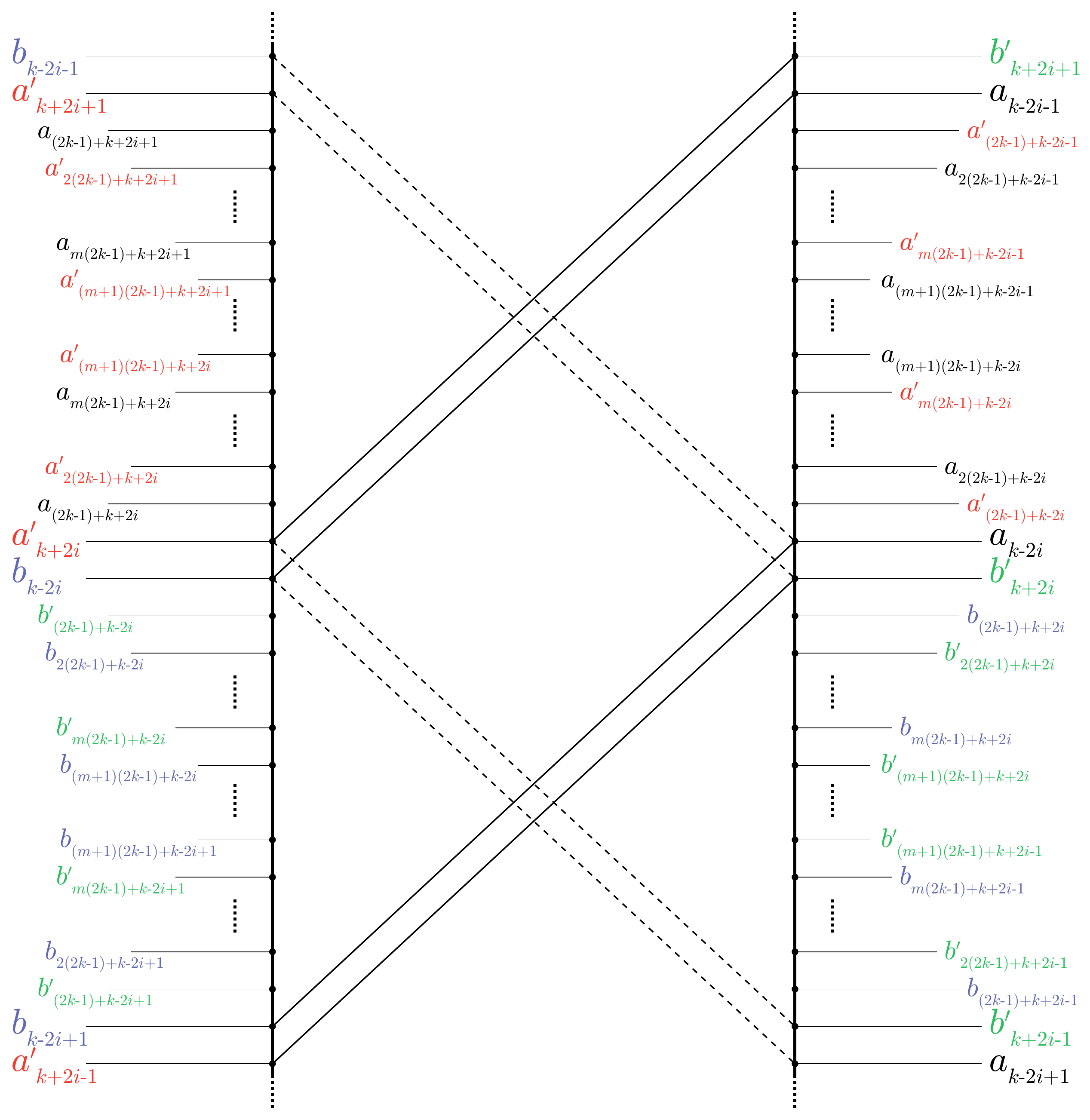}
\caption{The picture of a middle section of the disk for odd $k$.}
\label{sep_sphere_k_odd_middle.pdf}
\end{figure}

\begin{figure}[htbp]
\centering
\includegraphics[width=\textwidth]{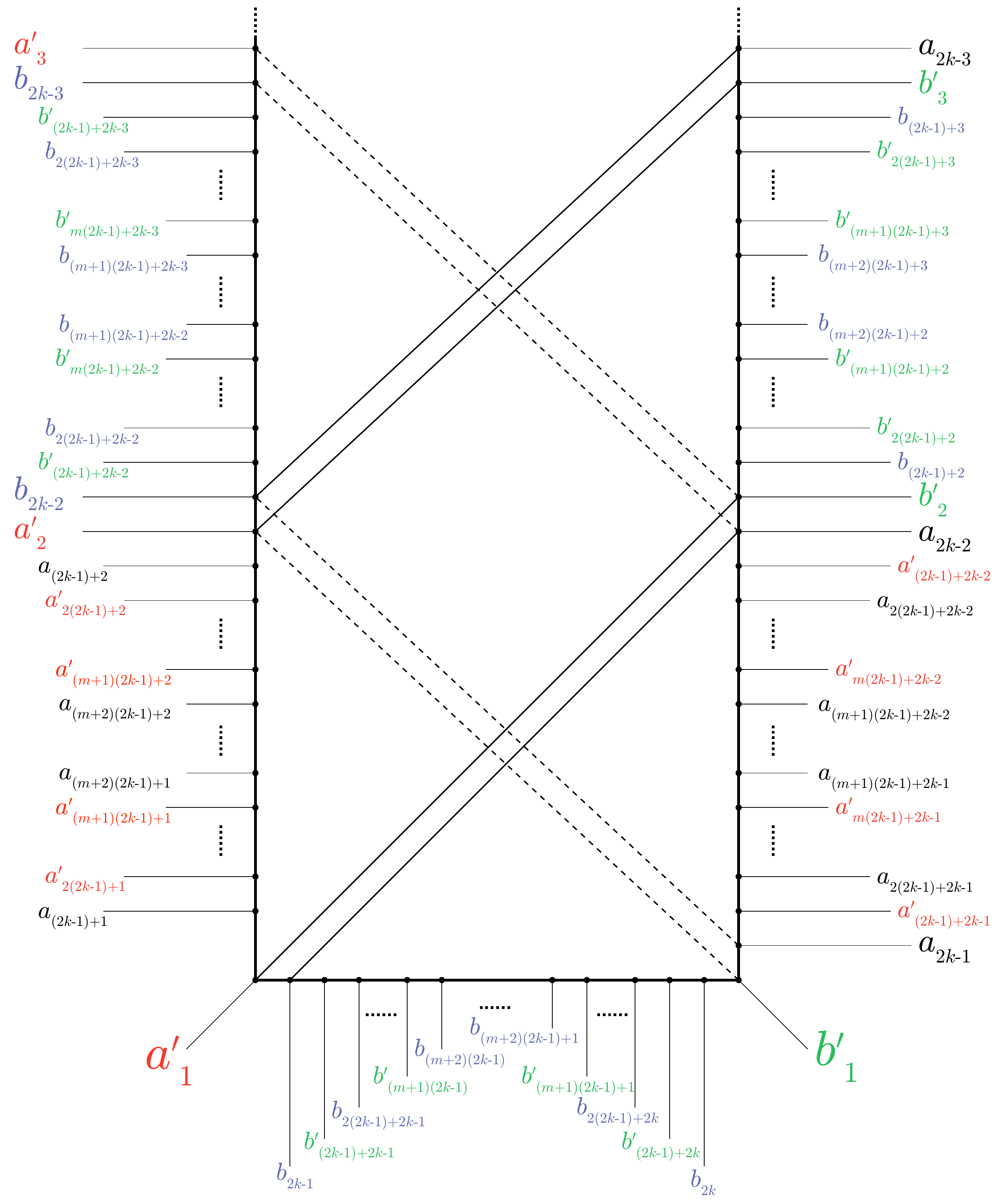}
\caption{The picture at the bottom of the disk for odd $k$. }
\label{sep_sphere_k_odd_bottom.pdf}
\end{figure}

In figure \ref{sep_sphere_diagram_example2.pdf} we modify the picture of the disk again to a rectangular form. Here $a_1, a'_1, b_1$ and $b'_1$ are at the corners of the rectangle so the  solid lines cross over the diagonal line from top left to bottom right, and dashed lines cross over the diagonal from top right to bottom left. In the left diagram we have added up until $\alpha_3$ and $\beta_3$, where we again have a form of {\bf changeover} analogous to in the non-separating case: we cross the band end $a'_3$, of $\alpha_3$, with the band end $b_3$, of $\beta_3$. Again, we can continue the pattern before the changeover, back and forth from the top and from the bottom as many times as we want, without meeting in the middle until, in this example, the third step. As in the non-separating case, we add genus up until the changeover, and add boundary components afterwards, again by adding pairs of bands. In general, the changeover occurs at step $k\geq 2$, where we follow a similar pattern: when adding the bands $\alpha_k$ and $\beta_k$ we cross the band end $a'_k$ of $\alpha_k$ with the band end $b_k$ of $\beta_k$. After the changeover we no longer have a choice for where the ends of the bands are to be added and must follow the pattern. See figures  \ref{sep_sphere_k_odd_top.pdf}, \ref{sep_sphere_k_odd_middle.pdf} and \ref{sep_sphere_k_odd_bottom.pdf} for the general case when $k$ is odd. When $k$ is even we have a similar situation, in fact we can convert the figure to that case by reflecting the top and bottom end diagrams along a vertical axis, swapping solid and dashed lines, and swapping all $a$s with $b$s. The indices of the labels, whether or not they have a prime, and the entire middle diagram are all unchanged.\\

We can continue adding pairs of bands in the pattern shown in figures \ref{sep_sphere_k_odd_top.pdf}, \ref{sep_sphere_k_odd_middle.pdf} and \ref{sep_sphere_k_odd_bottom.pdf} indefinitely, but for the construction of a particular surface we stop at some point. Let $S_{n,k}$ be the surface described above using a total of $n$ pairs of bands and doing the changeover at step $k\geq 2$. At the very end of the construction we may need to add only one of the two next bands in the same pattern as the general construction, and we let $S_{n,k}'$ to be the surface obtained from $S_{n,k}$ by adding the extra band $\alpha_{n+1}$ to $S_{n,k}$.

\begin{prop}\label{sep genus and boundary}
The surface $S_{n,k}$ ($S_{n,k}'$) has genus $g=k-1$ and $2n-2g+1$ (resp., $2n-2g+2$) boundary components.
\end{prop}
\begin{proof}
Denote by $\Sigma_s$ the surface obtained from adding the first $s$ steps (pairs of bands) in the construction of $S_{n,k}$. As in the proof of Proposition \ref{or sep genus and boundary} all bands added have an even number of half twists.
The surface $\Sigma_0$ is the starting disk, so $|\partial\Sigma_0|=1$. Let $i$  be such that $i< k-1$. Assume that $|\partial\Sigma_i|=1$. The next band added is $\alpha_{i+1}$. As can be checked by construction, the band $\beta_{i+1}$ has its ends on the two boundary components defined by $\alpha_{i+1}$. Then $|\partial\Sigma_{i+1}|=1$. By induction on $i$, we have $|\partial\Sigma_{k-1}|=1$.\\

At step $k$ we do the changeover. By the construction of $S_{n,k}$, we can check that the bands ends of $\alpha_{k},\ldots,\alpha_n$ and of $\beta_{k},\ldots,\beta_n$ are on the same boundary component. This means, by Lemma \ref{adding bands}(b), that each of these bands added increases the number of boundary components by one. Therefore, the number of boundary components of $S_{n,k}$ is $2n-2(k-1)+1$. As the Euler characteristics of $S_{n,k}$ is $1-2n$ the surface $S_{n,k}$ has genus $k-1$.\\

The surface $S_{n,k}'$ is obtained from $S_{n,k}$ by adding the band $\alpha_{n+1}$ which increases the number of boundary components by one.
\end{proof}

\begin{figure}[htb]
\centering
\includegraphics[width=\textwidth]{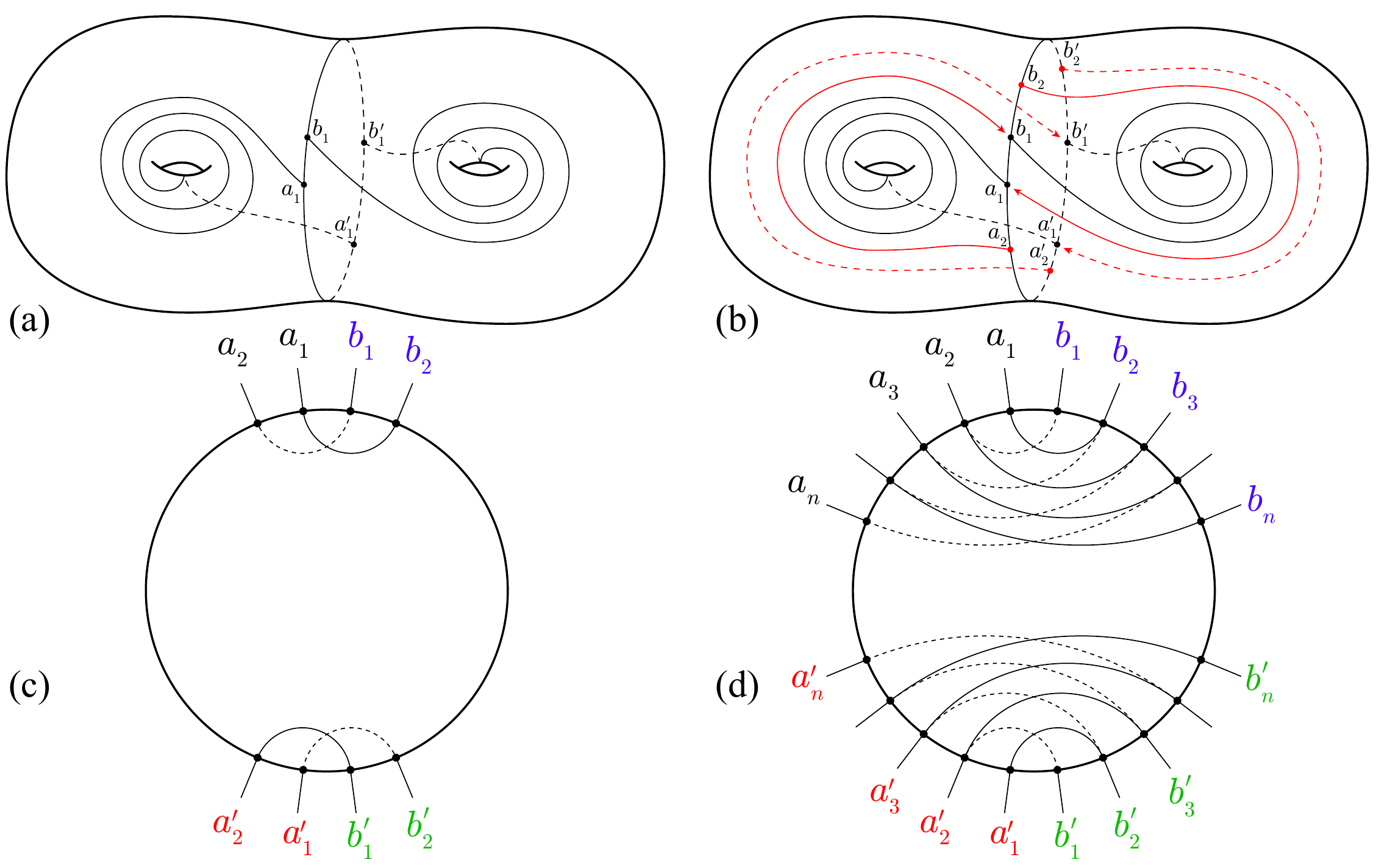}
\caption{Diagrams for the construction of separating surfaces of genus zero.}
\label{handlebody_sep_genus0.pdf}
\end{figure}

In this way we can get any genus $g \geq 1$ and any number of boundary components $b \geq 1$. To get genus $g=0$ we need a slightly different starting configuration. See figure \ref{handlebody_sep_genus0.pdf}(a). We start with the same disk as in the general separating case and the first pair of bands, $\alpha_1$ and $\beta_1$ again go three times around the $x$ and $y$ handles. However the ends do not alternate around the disk: $a_1$ and $a'_1$ are adjacent. The pattern continues in figure \ref{handlebody_sep_genus0.pdf}(b), with schematic diagrams in figure \ref{handlebody_sep_genus0.pdf}(c,d) analogous to figure \ref{handlebody_sep_disk_step2.pdf}. We can continue indefinitely and there are never any choices to make. We denote the surface obtained by adding $n$ bands following this procedure by $S_{n,0}$ (and, by $S_{n,0}'$ the surface obtained from $S_{n,0}$ by adding the band $\alpha_{n+1}$).\\

\begin{prop}
The surface $S_{n,0}$ ($S_{n,0}'$) has genus $0$ and $2n+1$ (resp., $2n+2$) boundary components.
\end{prop} 
\begin{proof}
The argument is similar to the proof of Proposition \ref{sep genus and boundary}.
\end{proof}

Suppose that $i:S_{n,k}\rightarrow H_2$ is the inclusion map as constructed above and $i_*: \pi_1(S_{n,k})\rightarrow \pi_1(H_2)$ the induced map by $i$. As mentioned before, we denote by $x$ and $y$ two generators of $\pi_1(H_2)$. Following the surface construction the generators of $\pi_1(S_{n,k})$ are $\alpha_1,\alpha_2,\ldots ,\alpha_{n},\beta_1, \beta_2,\ldots,\beta_{n} $.\\

\begin{align*}
\intertext{If $k= 0$ we have:}
i_*(\alpha_1)&=x^3\\
i_*(\beta_1)&=y^3\\
i_*(\alpha_2)&=x^{-1}i_*(\beta_1)x\\
i_*(\beta_2)&=y^{-1}i_*(\alpha_1)y\\
&\vdots\\
i_*(\alpha_n)&=x^{-1}i_*(\beta_{n-1})x\\
i_*(\beta_n)&=y^{-1}i_*(\alpha_{n-1})y\\
\intertext{If $k\geq 2$ assume that $n=q(2k-1)+r$, where $0\leq r< 2k-1$. When $k$ is odd (similarly for even $k$) we have:}
i_*(\alpha_1)&=x^3\\
i_*(\beta_1)&=y^3\\
i_*(\alpha_2)&=xi_*(\beta_1)x\\
i_*(\beta_2)&=yi_*(\alpha_1)y\\
&\vdots\displaybreak[0]\\
i_*(\alpha_{2k-2})&=xi_*(\beta_{2k-3})x\\
i_*(\beta_{2k-2})&=yi_*(\alpha_{2k-3})y\\
i_*(\alpha_{2k-1})&=(xy)^{k-1}x^3(yx)^{k-1}\\
i_*(\beta_{2k-1})&=(yx)^{k-1}y^3(xy)^{k-1}\\
i_*(\alpha_{(2k-1)+1})&=x^{-1}i_*(\beta_{2k-1})x\\
i_*(\beta_{(2k-1)+1})&=yi_*(\alpha_{2k-1})y^{-1}\\
i_*(\alpha_{(2k-1)+2})&=x^{-1}i_*(\beta_{(2k-1)+1})x^{-1}\\
i_*(\beta_{(2k-1)+2})&=y^{-1}i_*(\alpha_{(2k-1)+1})y^{-1}\\
&\vdots\displaybreak[0]\\ 
i_*(\alpha_{(2k-1)+2k-2})&=x^{-1}i_*(\beta_{(2k-1)+2k-3})x^{-1}\\
i_*(\beta_{(2k-1)+2k-2})&=y^{-1}i_*(\alpha_{(2k-1)+2k-3})y^{-1}\\
i_*(\alpha_{2(2k-1)})&=(x^{-1}y^{-1})^{k-1}x^{-1}(yx)^{k-1}y^3(xy)^{k-1}x(y^{-1}x^{-1})^{k-1}\\
i_*(\beta_{2(2k-1)})&=(y^{-1}x^{-1})^{k-1}y(xy)^{k-1}x^3(yx)^{k-1}y^{-1}(x^{-1}y^{-1})^{k-1}\\
i_*(\alpha_{2(2k-1)+1})&=x^{-1}i_*(\beta_{2(2k-1)})x\\
i_*(\beta_{2(2k-1)+1})&=yi_*(\alpha_{2(2k-1)})y^{-1}\\
i_*(\alpha_{2(2k-1)+2})&=xi_*(\beta_{2(2k-1)+1})x\\
i_*(\beta_{2(2k-1)+2})&=yi_*(\alpha_{2(2k-1)+1})y\\
&\vdots\displaybreak[0]\\
i_*(\alpha_{(2j-2)(2k-1)+2k-2})&=xi_*(\beta_{(2j-2)(2k-1)+2k-3})x\\
i_*(\beta_{(2j-2)(2k-1)+2k-2})&=yi_*(\alpha_{(2j-2)(2k-1)+2k-3})y\\
i_*(\alpha_{(2j-1)(2k-1)})&=(xy)^{k-1}[x^{-1}(y^{-1}x^{-1})^{k-1}y(xy)^{k-1}]^{j-1}x^3\\
&\quad [(yx)^{k-1}y^{-1}(x^{-1}y^{-1})^{k-1}x]^{j-1}(yx)^{k-1}\\
i_*(\beta_{(2j-1)(2k-1)})&=(yx)^{k-1}[y(x^{-1}y^{-1})^{k-1}x^{-1}(yx)^{k-1}]^{j-1}y^3\\
&\quad[(xy)^{k-1}x(y^{-1}x^{-1})^{k-1}y^{-1}]^{j-1}(xy)^{k-1}\\
i_*(\alpha_{(2j-1)(2k-1)+1})&=x^{-1}i_*(\beta_{(2j-1)(2k-1)})x\\
i_*(\beta_{(2j-1)(2k-1)+1})&=yi_*(\alpha_{(2j-1)(2k-1)})y^{-1}\\
i_*(\alpha_{(2j-1)(2k-1)+2})&=x^{-1}i_*(\beta_{(2j-1)(2k-1)+1})x^{-1}\\
i_*(\beta_{(2j-1)(2k-1)+2})&=y^{-1}i_*(\alpha_{(2j-1)(2k-1)+1})y^{-1}\\
&\vdots\displaybreak[0]\\
i_*(\alpha_{(2j-1)(2k-1)+2k-2})&=x^{-1}i_*(\beta_{(2j-1)(2k-1)+2k-3})x^{-1}\\
i_*(\beta_{(2j-1)(2k-1)+2k-2})&=y^{-1}i_*(\alpha_{(2j-1)(2k-1)+2k-3})y^{-1}\\
i_*(\alpha_{(2j)(2k-1)})&=(x^{-1}y^{-1})^{k-1}x^{-1}(yx)^{k-1}[y(x^{-1}y^{-1})^{k-1}x^{-1}(yx)^{k-1}]^{j-1}y^3\\
&\quad[(xy)^{k-1}x(y^{-1}x^{-1})^{k-1}y^{-1}]^{j-1}(xy)^{k-1}x(y^{-1}x^{-1})^{k-1}\\
i_*(\beta_{(2j)(2k-1)})&=(y^{-1}x^{-1})^{k-1}y(xy)^{k-1}[x^{-1}(y^{-1}x^{-1})^{k-1}y(xy)^{k-1}]^{j-1}x^3\\
&\quad[(yx)^{k-1}y^{-1}(x^{-1}y^{-1})^{k-1}x]^{j-1}(yx)^{k-1}y^{-1}(x^{-1}y^{-1})^{k-1}\\
i_*(\alpha_{(2j)(2k-1)+1})&=x^{-1}i_*(\beta_{(2j)(2k-1)})x\\
i_*(\beta_{(2j)(2k-1)+1})&=yi_*(\alpha_{(2j)(2k-1)})y^{-1}\\
i_*(\alpha_{(2j)(2k-1)+2})&=xi_*(\beta_{(2j)(2k-1)+1})x\\
i_*(\beta_{(2j)(2k-1)+2})&=yi_*(\alpha_{(2j)(2k-1)+1})y\\
&\vdots\displaybreak[0] \\
\intertext{Either $q$ is odd and we have:}
&\vdots\\
i_*(\alpha_{(q-1)(2k-1)+2k-2})&=xi_*(\beta_{(q-1)(2k-1)+2k-3})x\\
i_*(\beta_{(q-1)(2k-1)+2k-2})&=yi_*(\alpha_{(q-1)(2k-1)+2k-3})y\\
i_*(\alpha_{q(2k-1)})&=(xy)^{k-1}[x^{-1}(y^{-1}x^{-1})^{k-1}y(xy)^{k-1}]^{\frac{q-1}{2}}x^3\\
&\quad[(yx)^{k-1}y^{-1}(x^{-1}y^{-1})^{k-1}x]^{\frac{q-1}{2}}(yx)^{k-1}\\
i_*(\beta_{q(2k-1)})&=(yx)^{k-1}[y(x^{-1}y^{-1})^{k-1}x^{-1}(yx)^{k-1}]^{\frac{q-1}{2}}y^3\\
&\quad[(xy)^{k-1}x(y^{-1}x^{-1})^{k-1}y^{-1}]^{\frac{q-1}{2}}(xy)^{k-1}\\
i_*(\alpha_{q(2k-1)+1})&=x^{-1}i_*(\beta_{q(2k-1)})x\\
i_*(\beta_{q(2k-1)+1})&=yi_*(\alpha_{q(2k-1)})y^{-1}\\
i_*(\alpha_{q(2k-1)+2})&=x^{-1}i_*(\beta_{q(2k-1)+1})x^{-1}\\
i_*(\beta_{q(2k-1)+2})&=y^{-1}i_*(\alpha_{q(2k-1)+1})y^{-1}\\
&\vdots\\
i_*(\alpha_{q(2k-1)+r})&=x^{-1}i_*(\beta_{q(2k-1)+(r-1)})x^{-1}\\
i_*(\beta_{q(2k-1)+r})&=y^{-1}i_*(\alpha_{q(2k-1)+(r-1)})y^{-1}\\
\displaybreak[0]\\
\intertext{Or $q$ is even and we have:}
&\vdots\\
i_*(\alpha_{(q-1)(2k-1)+2k-2})&=xi_*(\beta_{(q-1)(2k-1)+2k-3})x\\
i_*(\beta_{(q-1)(2k-1)+2k-2})&=yi_*(\alpha_{(q-1)(2k-1)+2k-3})y\\
i_*(\alpha_{q(2k-1)})&=(x^{-1}y^{-1})^{k-1}x^{-1}(yx)^{k-1}[y(x^{-1}y^{-1})^{k-1}x^{-1}(yx)^{k-1}]^{\frac{q}{2}-1}y^3\\
&\quad[(xy)^{k-1}x(y^{-1}x^{-1})^{k-1}y^{-1}]^{\frac{q}{2}-1}(xy)^{k-1}x(y^{-1}x^{-1})^{k-1}\\
i_*(\beta_{q(2k-1)})&=(y^{-1}x^{-1})^{k-1}y(xy)^{k-1}[x^{-1}(y^{-1}x^{-1})^{k-1}y(xy)^{k-1}]^{\frac{q}{2}-1}x^3\\
&\quad[(yx)^{k-1}y^{-1}(x^{-1}y^{-1})^{k-1}x]^{\frac{q}{2}-1}(yx)^{k-1}y^{-1}(x^{-1}y^{-1})^{k-1}\\
i_*(\alpha_{q(2k-1)+1})&=x^{-1}i_*(\beta_{q(2k-1)})x\\
i_*(\beta_{q(2k-1)+1})&=yi_*(\alpha_{q(2k-1)})y^{-1}\\
i_*(\alpha_{q(2k-1)+2})&=xi_*(\beta_{q(2k-1)+1})x\\
i_*(\beta_{q(2k-1)+2})&=yi_*(\alpha_{q(2k-1)+1})y\\
&\vdots\\
i_*(\alpha_{q(2k-1)+r})&=xi_*(\beta_{q(2k-1)+(r-1)})x\\
i_*(\beta_{q(2k-1)+r})&=yi_*(\alpha_{q(2k-1)+(r-1)})y\\
\end{align*}\\
 
For the surface $S_{n,k}'$ we continue by adding the word $i_*(\alpha_{n+1})$ similarly.\\
 
\begin{lemma}\label{word lemma 2}
Let $F_2$ be the free group with two generators, $g$ and $h$, with identity element  $e$. Consider 

$$\begin{array}{lcl}
A&=&\{v_1g^pv_1', v_2g^pv_2',\ldots, v_{n_1}g^pv_{n_1}'\}\\
B&=&\{u_1h^pu_1', u_2h^pu_2',\ldots, u_{n_2}h^pu_{n_2}'\} \,\text{, with $p\geq 3$,}\\
\end{array}$$
where $v_ig^pv_i'$, for $i=1,\ldots,n_1$, and $u_ih^pu'_i$, for $i=1,\ldots, n_2$ are reduced words on $g$ and $h$,  and $v_i,\,v_i',u_i \text{ and }u'_i$ are words where $g$, or its inverse, alternates with $h$, or its inverse.  Furthermore, $v_i$ ($u_i$) ends with $h$ or $h^{-1}$, (resp., $g$ or $g^{-1}$ ), and $v_i'$ ($u'_i$) starts with $h$ or $h^{-1}$ (resp., $g$ or $g^{-1}$). Assume also that $L(v_{i+1})>L(v_i)$, $L(v_{i+1}')>L(v_i')$, for $i=1,2,\ldots,n_1-1$ and $L(u_{i+1})>L(u_i)$, $L(u_{i+1}')>L(u_i')$, for $i=1,2,\ldots,n_2-1$. Take $w$ to be a reduced word on the elements of $A$ or $B$. That is, $w$ is expressed as a product of elements of $A$ or $B$, or inverses of those elements, with no consecutive cancelling terms. Then $w\neq e$. 

\end{lemma}
\begin{proof}
Let $w=w_1w_2\cdots w_s$ where each $w_i$, for $i=1,\ldots, s$ is a reduced word on either the elements of $A$ or the elements of $B$, with the condition that if $w_i$ is a word on the elements of $A$ (or $B$) then $w_{i+1}$ is a word on the elements of $B$ (resp., $A$). Note that from Lemma \ref{word lemma} we have $w_i\neq e$, for $i=1,\ldots, s$.\\ 

If $s=1$ then $w=w_1$ and $w\neq e$. Let $s\geq 2$. After reducing the $w_i'$ as words of $g$ and $h$, from Lemma \ref{word lemma} we have either: 
\begin{eqnarray*}
w_i&=&c_ig^{p_i}d_ig^{p_i'}c_i' \text{ if  $w_i$ is a word on the elements of $A$, or}\\ w_i&=&c_ih^{p_i}d_ih^{p_i'}c_i' \text{ if  $w_i$ is a word on the elements of $B$,} 
\end{eqnarray*}
where $d_i$ is some reduced word on $g$ and $h$, and $c_i, c_i'$, or its inverses, are in $\{u_i, u_i', v_i, v_i'\}$ and also $|p_i|, |p_i'|\geq 2$, for $i=1,\ldots,s$. Then we get either:

\begin{eqnarray}
w_iw_{i+1}&=&c_ig^{p_i}d_ig^{p_i'}c_i'c_{i+1}h^{p_{i+1}}d_{i+1}h^{p_{i+1}'}c_{i+1}' \text{, or}\label{case 1}\\
w_iw_{i+1}&=&c_ih^{p_i}d_ih^{p_i'}c_i'c_{i+1}g^{p_{i+1}}d_{i+1}g^{p_{i+1}'}c_{i+1}'.\label{case 2}
\end{eqnarray}

  Let $l_i=c_i'c_{i+1}$. We also have, for $i=\{1,\ldots, s-1\}$,
 \begin{itemize}
 \item[a)] $l_i=e$ , when $c_i'$ and $c_{i+1}$ cancel each other; or\\
 \item[b)] $l_i=h^{\pm 1}\cdots h^{\pm 1}$, when part of $c_i'$ ($c_{i+1}$) cancels $c_{i+1}$ (resp., $c_i'$) in case \ref{case 1} (resp., case 2); or\\
 \item[c)] $l_i=g^{\pm 1}\cdots g^{\pm 1}$, when part of $c_i'$ ($c_{i+1}$) cancels $c_{i+1}$ (resp., $c_i'$) in case \ref{case 2} (resp., case 1) ; or\\
 \item[d)] $l_i=h^{\pm 1}\cdots g^{\pm 1}$, when neither $c_i'$ nor $c_{i+1}$ in case \ref{case 1} is canceled; or
 \item[e)] $l_i=g^{\pm 1}\cdots h^{\pm 1}$, when neither $c_i'$ nor $c_{i+1}$ in case \ref{case 2} is canceled. 
 \end{itemize}
Therefore, as $|p_i|, |p_i'|\geq 2$, for $i=1,\ldots,s$ the powers $p_i$ and $p_i'$ are never canceled. This implies that $L(w)>s$ and $w\neq e$.

\end{proof}
   
  \begin{prop}\label{orientable separating}
 The homomorphism $i_*:\pi_1(S_{n,k})\rightarrow \pi_1(H_2)$ is injective.
 \end{prop} 
 \begin{proof}
Take $$A=\{i_*(\alpha_{2i-1}) \mid 1\leq 2i-1\leq n, i\in \N\}\cup\{i_*(\beta_{2i}) \mid 2\leq 2i\leq n, i\in \N\}$$ and $$B=\{i_*(\beta_{2i-1}) \mid 1\leq2i-1\leq n, i\in \N\}\cup \{i_*(\alpha_{2i}) \mid 2\leq 2i\leq n, i\in \N\}.$$ Using the words from above we can write $i_*(\alpha_{2i-1})=v_{2i-1}x^3v_{2i-1}'$, $i_*(\beta_{2i})=v_{2i}x^3v_{2i}'$, and $i_*(\beta_{2i-1})=u_{2i-1}y^3u_{2i-1}'$, $i_*(\alpha_{2i})=u_{2i}y^3u_{2i}'$. We have that $v_{i}x^3v_{i}'$ and $u_iy^3u_{i}'$, for $i=1,\ldots,n$, are reduced words on $x$ and $y$, where the words $v_i$, $v_i'$, $u_i$ and $u_i'$ are words where $x$, or its inverse, alternate with $y$, or its inverse.  Furthermore, $v_i$ ($u_i$) ends with $y$, (resp., $x$),  and $v_i'$ ($u'_i$) starts with $y$ (resp., $x$). Also, $L(v_{i+1})>L(v_i)$, $L(v_{i+1}')>L(v_i')$, $L(u_{i+1})>L(u_i)$ and $L(u_{i+1}')>L(u_i')$, for $i=1,\ldots, n-1$. So, we are under the hypothesis of Lemma \ref{word lemma 2}.\\

Let $w=i_*(\alpha)$ for some $\alpha\in \pi_1(S_{n,k})$ with $\alpha\neq e$. We can write $w$ as a reduced word on the elements of $A\cup B$: $w=w_1w_2\cdots w_s$, for some $s\in\N$, where the $w_i$'s alternate between being words in either the elements of $A$ or the elements of $B$. Then by Lemma \ref{word lemma 2}, $w\neq e$.  Therefore $i_*$ is injective.
 \end{proof} 
  
 The same argument, with the extra word $i_*(\alpha_{n+1})$ included, proves that  $i_*:\pi_1(S_{n,k}')\rightarrow \pi_1(H_2)$ is injective. A similar argument also works for the case when $k$ is even.\\
 
The above construction and results, and in particular Proposition \ref{orientable separating}, prove Theorem \ref{sep orientable in H2}.

\section{Incompressible surfaces in boundary reducible 3-manifolds}
\label{b-reducible manifold}

Consider a closed surface $F_-$, not necessarily connected. Take $F_-\times [0,1]$ and add a finite number of  1-handles to $F_-\times \{1\}$. A 3-manifold obtained in this way is said to be a compression body. Dually a connected compression body can be obtained from a connected surface $F_+$, by taking $F_+\times [0,1]$ and adding a finite number of 2-handles to $F_+\times \{0\}$.\\

Let $M$ be a compact 3-manifold with a compressible boundary component $B$ of genus greater or equal than two.\\

Take the compressible boundary component $B$  of $M$ and consider $B\times [0,1]\subset M$ where we view $B$ as $B\times \{1\}$. As $B$ is compressible there is a compressing disk for $B\times \{0\}$ in $M$. We say that a (disjoint) collection of such compressing disks $\{D_1, D_2,\ldots, D_n\}$, for some $n\in\mathbb{N}$, is maximal if any other compressing disk for $B\times \{0\}$ in $M$ either intersects or is isotopic to some disk in the collection\footnote{A maximal collection of compressing disks always exists for $B$ because $B$ is compressible in $M$ (and has finite genus).}.  Consider the compression body obtained from $B\times [0,1]$ by adding the 2-handles defined by the compressing disks  in a maximal collection of compressing disks for $B\times\{0\}$ in $M$. We denote this compression body by $C$, $\partial_+C=B$ and $\partial_-C=\partial C -B$. Let $N$ be defined by $M=C\cup_{\partial_-C}N$. We have that $\partial_-C$ is incompressible in $N$, otherwise some isotopy class of compressing disks for $B\times \{0\}$ would be missing from the construction of $C$. In addition $\partial_-C$ is incompressible in $C$, because $C$ is obtained from $\partial_-C\times [0,1]$ by adding a collection of 1-handles to $\partial_-C\times \{1\}$. This means that there is no compressing disk for $\partial_-C$ in $C$. We refer to such a compression body as a maximal compression body of $B$ in $M$.\\

Consider a connected compression body $C$ where $\partial_+C$ has genus greater than or equal to two and such that $\partial _+C$ is compressible in $C$. From now on $C$ will denote a compression body with these properties.\\

Let $T_s$ be a punctured torus with $s$ punctures and take $P_s=T_s\times [0,1]$ and denote by $Q$ a solid torus.\\

\begin{figure}[htb]
\centering
\includegraphics[width=\textwidth]{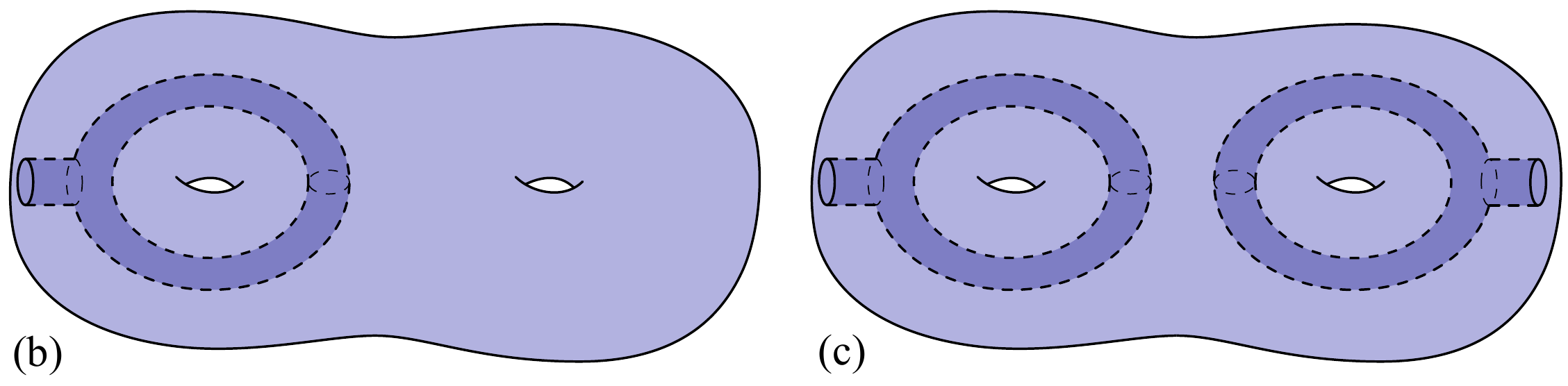}
\caption{Examples of the subspace $H$ in cases (b) ($H=P_1\cup_D Q$) and (c) ($H=P_1\cup_D P_1$) of lemma \ref{compression body decomposition}.}
\label{H2_variants.pdf}
\end{figure}

\begin{lemma}\label{compression body decomposition}
There are subspaces $H,\,C^*\subset C$ giving the decomposition $C=C^*\cup_J H$, where $C^*\cap H=J$ is a collection of annuli or disks properly embedded in $C$, such that
\begin{itemize}
\item[(a)] $H=Q\cup_DQ=H_2$ or,\\
\item[(b)] $H=P_{s_1}\cup_D Q$ or,\\
\item[(c)] $H=P_{s_1}\cup_D P_{s_2}$,
\end{itemize} 
for some $s_i\in\{0,1\}$, $i=1,2$, $D$ is a separating disk in $H$ and $C$, and $J$ is one disk or one annulus in case (a) and (b) and a collection of two disks or annuli in case (c).
\end{lemma}
\begin{proof}
Let us take $C$ obtained from a closed surface $F\times [0,1]$ by adding a collection of 1-handles to $F\times \{1\}$. Consider a graph $G$ where the vertices correspond to the components of $F\times [0,1]$ and the edges correspond to the 1-handles. Here we assume that $C$ is connected and so $G$ is a connected graph.\\      

Assume first that $G$ has two or more cycles. We can slide the ends of two 1-handles along $F\times \{1\}$ and other 1-handles until they are both on the same component of $F\times [0,1]$. Then by cutting along a disk $J$ we get $C=C^*\cup_J H$ where $H$ is a handlebody of genus two. Note that in this case $H$ contains an essential compressing disk that is separating, and an essential compressing disk that is non-separating, in both $H$ and $C$.\\

Now assume that $G$ has precisely one cycle. If all components of $F$ are spheres then the genus of $\partial_+ C$ is one. This is a contradiction to the assumption that $\partial _+ C$ has genus greater than or equal to two. Therefore at least one component of $F$ has genus greater than or equal to one. 
\begin{itemize}
\item If there is a component of $F\times [0,1]$ with genus one then we slide one 1-handle to be a loop on this vertex of $G$ and if necessary, slide the other 1-handles around (altering the corresponding graph $G$) to leave only one 1-handle connecting this component to the other components. So, now there is a disk $J$ in $C$ such that $C=C^*\cup_J H$,where $H=P_0\cup_D Q$.
\item If, on the other hand, there is no component of $F$ with genus one then take a component  $F_1$ with positive genus. We slide one 1-handle to be a loop on this vertex of $G$ and again leave only one 1-handle connecting this component to the other components. We take a simple closed curve $a$ in $F_1$ separating a punctured torus from the rest of $F_1$ and we slide both ends of the 1-handle to this punctured torus. Taking $F\times [0,1]$ we get a separating annulus $A=a\times [0,1]$. By cutting along this annulus we get $C=C^*\cup_J H$ where  $H=P_1\cup_D Q$ and $J=A$.
\end{itemize}
Note that also in this case $H$ contains an essential compressing disk that is separating, and an essential compressing disk that is non-separating, in both $H$ and $C$.\\

Finally, assume that $G$ is a tree. If the components of $F$ are all spheres apart possibly from at most one torus component then the genus of $\partial_+ C$ is less or equal than one. As in the previous case, this gives a contradiction to the assumption that $\partial _+ C$ has genus greater than or equal to two.  If the components of $F$  are all spheres apart from a single surface of genus $g\geq 2$ then the genus of $\partial_+ C$ is also $g\geq 2$. Then $\partial_+C$ is incompressible in $C$, as the existence of a compressing disk would mean that no genus $g$ component of $F$ exists. So, $F$ must have more than one component of genus greater than or equal to one.  Take two of these components, say $F_1$ and $F_2$, and slide 1-handles if necessary until each $F_i\times [0,1]$ is adjacent to two edges in $G$ and one edge connects $F_1\times [0,1]$ to $F_2\times [0,1]$. We have different cases depending on if $F_1$ and $F_2$ are tori or are surfaces of higher genus: 
\begin{itemize}
\item If $F_1$ and $F_2$ are both tori then cut the two 1-handles attached to each of them that does not connect them to each other. We denote by $H$ the component with the 1-handle attaching $F_1$ and $F_2$. So, we get $C=C^*\cup_J H$ where $H=P_{0}\cup_D P_{0}$ and $J$ is a collection of two disks. 
\item Without loss of generality, we assume $F_1$ is a torus and $F_2$ has genus greater than or equal to two. We cut the 1-handle attached to $F_1$ but not to $F_2$ along a disk. Consider a simple closed curve $a$ in $F_2$ cutting a punctured torus from $F_2$. We slide the 1-handle connecting $F_1$ and $F_2$ to this punctured torus and if necessary, slide other 1-handles until there are no other 1-handles attached to the punctured torus. Taking $F \times [0,1]$ we get an annulus $A=a\times [0,1]$. By cutting along this annulus and the disk we get $C=C^*\cup_J H$ where $H=P_0\cup_D P_1$ and $J$ is a collection of a disk and an annulus. 
\item Finally, assume that both $F_1$ and $F_2$ have genus greater than or equal to two. Take simple closed curves $a_i$ in $F_i$ cutting a punctured torus from $F_i$, $i=1,2$, and we slide the 1-handle connecting $F_1$ and $F_2$ to these punctured tori and all other 1-handles off of the punctured tori. By taking $F_\times [0,1]$ we get the annuli $A_i=a_i\times [0,1]$, for $i=1,2$. We cut along this annuli and we get $C=C^*\cup_J H$ where $H=P_1\cup_D P_1$ and $J=A_1\cup A_2$.
\end{itemize}
Note that in this case $H$ contains an essential compressing disk that is separating in both $H$ and $C$ (and no essential non-separating, compressing disk in either $H$ or $C$).\\
\end{proof} 

Let $S$ be a surface embedded in $H_2=Q\cup_D Q$ as in the previous sections. Let $E$ be the starting disk of the construction of the embedding of $S$ in $H_2$. Then, by construction, $S$ is in a regular neighborhood   $E\times [0,1]\cup \partial H_2\times [0,1]$ of $E\cup \partial H_2$. We know that the embedding of $S$ is constructed from a finite collection of bands parallel to $\partial H_2$ attached to $E$. So, we can take a collection of points $\{x_1, x_2\}\subset \partial H_2-\partial E$ such that $S\subset E\times[0,1] \cup \partial H_2\times [0,1]-x_i\times [0,1]$, for $i=1,2$. Therefore, in the same way we constructed the $\pi_1$-injective embedding of $S$ in $H_2$ we construct a $\pi_1$-injective embedding of $S$ in $H=P_{s_1}\cup_D Q$ and also in $H=P_{s_1}\cup_D P_{s_2}$, for $s_i\in\{0,1\}$, $i=1,2$. We will prove the $\pi_1$-injectivity in Lemma \ref{pi1 injective compression body}. Note that if $H$ is as in Lemma \ref{compression body decomposition} (a) or (b) we can construct  the  embedding of $S$ in $H$ following any of the previous sections as $H$ has both an essential separating disk and an non-separating disk; if $H$ is as in Lemma \ref{compression body decomposition} (c) we can only construct $S$ following section \ref{separating case} as $H$ has an essential separating disk but no non-separating disks.\\

Let $i:S\rightarrow H$ be the embedding discussed above and $l:H\rightarrow H_2$ the inclusion map. By construction we have the commutativity of the following diagram:
$$\xymatrix{&H_2\\
S\ar[ur]^{i}\ar[r]_{j}&H\ar[u]_{l}}\\$$

Let $C$ be a compression body as referred above. As in Lemma \ref{compression body decomposition} we have the decomposition $C=C^*\cup_J H$, where $C^*\cap H=J$ and $J$ is a collection of disks and annuli where at most two components are annuli. Therefore, $\pi_1(C)=\pi_1(C^*)*_{\pi_1(J)}\pi_1(H)$ is a free product with amalgamation along the group $\pi_1(J)$, which is the trivial group, $\Z$ or $\Z*\Z$.\\

Consider the inclusion $h:H\rightarrow C^*\cup_J H$. Then we have a proper embedding of $S$ in $C$, induced by $j$, given by $g=h\circ j$ where the boundary of $S$ lies in $\partial_+C$.\\ 

\begin{prop}\label{pi1 injective compression body}
The induced homomorphism $g_*:\pi_1(S)\rightarrow \pi_1(C)$ is injective.
\end{prop}
\begin{proof}
We know that $i=l\circ j$. Therefore, for the induced homomorphisms we have commutativity of the following diagram:
$$\xymatrix{&\pi_1(H_2)\\
\pi_1(S)\ar[ur]^{i_*}\ar[r]_{j_*}&\pi_1(H)\ar[u]_{l_*}}\\$$

Let $\overline{C}$ be obtained from $C$ by adding a 2-handle along each core of the annulus components of $J$. Then $\overline{C}=\overline{C^*}\cup_{\overline{J}}\overline{H}$, where $\overline{C^*}$ and $\overline{H}$ are obtained from $C^*$ and $H$ by the addition of the 2-handles along annulus components of $J$, and $\overline{J}$ is one or two separating disks in $\overline{C}$. Consider $k:C\rightarrow \overline{C}$  the inclusion map. The induced homomorphism $k_*:\pi_1(C)\rightarrow \pi_1(\overline{C})$ is onto. Let $\overline{l}:\overline{H}\rightarrow H_2$ also be the inclusion map. So the inclusion map $l$ of $H$ in $H_2$ can be written as the composition of $\overline{l}$ and the inclusion of $H$ in $\overline{H}$. Then, as $\pi_1(\overline{C})=\pi_1(\overline{C^*})*\pi_1(\overline{H})$ the homomorphism $\overline{f}_*:\pi_1(\overline{C})\rightarrow \pi_1(H_2)$ defined by 

\begin{equation*}
\overline{f}_*(a)=\left\{
\begin{array}{ll}
\overline{l}_*(a)&\text{ if } a\in \pi_1(\overline{H})\\
e &\text{ if } a\in \pi_1(\overline{C^*})
\end{array}\right.
\end{equation*}
where $e$ is the identity element of $\pi_1(H_2)$, is onto. Let $f_*=\overline{f}_*\circ k_*$. Consider also the inclusion maps $h:H\rightarrow C$ and $\overline{h}=k\circ h$. Therefore, the following diagram is commutative: 
$$\xymatrix{\pi_1(H_2)&\pi_1(\overline{C})\ar[l]_{\overline{f}_*}\\
\pi_1(H)\ar[u]^{l_*}\ar[r]_{h_*}\ar[ur]^{\overline{h}_*}&\pi_1(C)\ar[u]_{k_*} }$$

From the commutativity of these diagrams we have the commutativity of the following diagram:
$$\xymatrix{&\pi_1(H_2)&\\
\pi_1(S)\ar[ru]^{i_*}\ar[r]_{j*}\ar@/_1.5pc/[rr]_{g_*}&\pi_1(H)\ar[u]_{l_*}\ar[r]_{h_*}&\pi_1(C)\ar[lu]_{f_*}}$$

As $f_*$ is surjective and $i_*$ is injective we have that $g_*=h_*\circ j_*$ is injective. This gives us the conclusion of the proposition.
\end{proof}

Consider the inclusion map $q:C\rightarrow M$ and also the embedding $p: S\rightarrow M$, where $p=q\circ g$.

\begin{lemma}\label{comp body injects into M}
The induced homomorphism $q_*:\pi_1(C)\rightarrow\pi_1(M)$ is injective.
\end{lemma}
\begin{proof}
Assume $M$ is connected and let $N$ be defined by $M=C\cup_S N$. $C$ is connected but $N$ may not be, and we let $N_1, N_2, \ldots, N_n$ be the connected components of $N$. The components $S_1, S_2, \ldots, S_m$ of $S$ are incompressible boundary components of both $C$ and the $N_i$. We build $M$ from $C$ and the $N_i$ by gluing along the $S_j$'s. At each gluing, we apply either van Kampen's theorem (if this is the first time we join a particular $N_i$ to $C$) or perform an HNN extension (if this is not the first time we join a particular $N_i$ to $C$). In both cases the fundamental group of the component containing $C$ injects into the fundamental group of the resulting manifold after the gluing along $S_j$ (this uses the injectivity of $\pi_1(S_j)$ into both $\pi_1(C)$ and $\pi_1(N_i)$). Thus by induction, $\pi_1(C)$ injects into $\pi_1(M)$.
\end{proof}

\begin{prop}\label{separating}
The induced homomorphism $p_*:\pi_1(S)\rightarrow \pi_1(M)$ is injective.     
\end{prop}
\begin{proof}
We have from Proposition \ref{pi1 injective compression body} that $g_*$ is injective and from Lemma \ref{comp body injects into M} that $q_*$ is injective. These imply that $p_*$ is injective.
\end{proof}

\begin{proof}[Proof of Theorem \ref{all surfaces in M}]
Let $M$ be a 3-manifold and $B$ a compressible boundary component of genus greater than or equal to two, of type \emph{ns}. Consider a maximal compression body, $C$, in $M$ obtained from $B$. As in the proof of Lemma \ref{compression body decomposition}, we consider a graph $G$ where the vertices correspond to the components of $\partial_- C$, and the edges correspond to the 1-handles. Each compressing disk of $B$ in $M$ can be seen as the co-core of a 1-handle and (by possibly changing $G$) corresponds to a point in the corresponding edge of $G$.\\ 

Assume $B$ is of type \emph{ns}. So, it has a compressing disk $D$ in $M$ where $\partial D$ is non-separating in $B$, and so $D$ is necessarily non-separating in $M$. This implies that the point in $G$ (possibly changed by considering a different 1-handle decomposition of $C$) corresponding to $D$ is non-separating. Therefore, $G$ has at least one cycle. Then, from the proof of Lemma \ref{compression body decomposition}, $C$ has a decomposition as in Lemma \ref{compression body decomposition} (a) or (b). In both cases, $H$ has an essential disk that is separating, and an essential disk that is non-separating, in both $H$ and $M$. In fact, from Lemma \ref{compression body decomposition} (a) or (b) the disk $D$ (as denoted there) is also separating in $M$ and a non-separating disk of $Q$ (as denoted in the Lemma) is also non-separating in $M$. Let $D_1$ be a essential non-separating disk in $H$ and $M$.  Then we can construct embeddings of surfaces in $H$, and hence in $C$ and $M$, following section \ref{non-separating case} and the discussion in this section. Note that we start this construction in $H$ with a non-separating essential disk in $H$, with the same property in $C$ and $M$. So, we construct a non-separating embedding in $M$ for each orientable, and non-orientable, compact surface with boundary. Now, let $D_2$ be a essential separating disk in $H$ and $M$. Then we can construct embeddings of surfaces in $H$, and hence in $C$ and $M$, following section \ref{separating case} and the discussion in this section. As the starting disk is separating in $M$ we construct a separating embedding for all orientable, compact surfaces with boundary in $M$.  By Lemma \ref{separating} these embeddings are $\pi_1$-injective. This completes the proof of Theorem \ref{all surfaces in M}.\\
\end{proof}

\begin{proof}[Proof of Theorem \ref{sep orientable in M}]
Consider the compressible boundary components of $M$ with genus greater than or equal to two. If one of these components is of type \emph{ns} then we have Theorem \ref{all surfaces in M} and our result follows. Otherwise no boundary component is of type \emph{ns}. So all compressing disks $D$ of a boundary component $B$ in $M$ are such that $\partial D$ separates $B$. Fix a choice of compressing disk $D$ of a compressible boundary component $B$ of $M$. As before, let $C$ denote the maximal compression body defined by $B$ in $M$. As $B$ is not of type \emph{ns}, the compression body $C$ has the decomposition given by  Lemma \ref{compression body decomposition} (c). The disk $D$ is dual to a 1-handle in a 1-handle decomposition of $C$ and $D$ is a compressing disk of $\partial_+C$ so we can adjust the decomposition such that $D$ is the separating disk in $H$ from Lemma \ref{compression body decomposition} (c). We follow section \ref{separating case} and the discussion in this section to construct the embeddings of orientable surfaces in $H$. Although the surfaces constructed in this way are separating in $H$ and $C$ they are not necessarily separating in $M$. We have two cases. If $D$ is (non-) separating in $M$ then the surfaces are (resp., non-) separating in $M$. By Proposition \ref{separating} these embeddings are $\pi_1$-injective. Since the choice of $D$ was arbitrary, this completes the proof of Theorem \ref{sep orientable in M}.
\end{proof}

\vspace{1cm}
Jo\~ao Miguel Nogueira\\
Department of Mathematics, University of Texas at Austin, USA\\
CMUC, Department of Mathematics, University of Coimbra, Portugal\\ 
Email: jnogueira@math.utexas.edu, nogueira@mat.uc.pt\vspace{0.5cm}\\
Henry Segerman\\
Department of Mathematics and Statistics, The University of Melbourne, VIC 3010, Australia\\
Email: segerman@unimelb.edu.au\\


\begin{thebibliography}{99}
\bibitem{Eudave}
M. Eudave-Mu\~noz, Incompressible surfaces in tunnel number one knot complements, Topology and its Applications 98 (1999), 167-189.
\bibitem{Howards}
H. N. Howards, Generating disjoint incompressible surfaces,
preprint, 1-18.
\bibitem{Jaco}
W. Jaco, Lectures on three-manifold topology, CMBS 43 by Amer. Math. Soc., Providence, Rhode Island, 1997 reprint.  
\bibitem{Qiu 1}
R. Qiu, Incompressible surfaces in handlebodies and closed 3-manifolds of Heegaard genus 2,
Proc. Amer. Math. Soc. Vol. 128 N. 10 (2000), 3091-3097.
\bibitem{Qiu 2}
R. Qiu, Incompressible surfaces in $\partial$-reducible 3-manifolds,
unpublished, 1-13.
\end{thebibliography}
\end{document}